\documentclass{amsart}

\newtheorem{thm}{Theorem}[section]
\newtheorem{prop}[thm]{Proposition}
\newtheorem{cor}[thm]{Corollary}
\newtheorem{lemma}[thm]{Lemma}

\theoremstyle{definition}
\newtheorem{dfn}[thm]{Definition}

\theoremstyle{remark}
\newtheorem{rem}[thm]{Remark}

\numberwithin{equation}{section}

\newcommand{\claim}{\textsf{Claim}\, }

\newcommand{\grad}{\ensuremath{\mathrm{grad}\ }}

\newcommand{\tub}{\ensuremath{\mathrm{Tub} }}
\newcommand{\F}{\ensuremath{\mathcal{F}}}

\newcommand{\singularF}{\ensuremath{\mathcal{X}_{\mathcal{F}}}}

\newcommand{\singularFsigma}{\ensuremath{\mathcal{X}_{\mathcal{F}_{\Sigma}}}}

\newcommand{\metric}{\ensuremath{ g }}
\begin{document}

\title[Desingularization of s.r.f]{Desingularization of singular Riemannian foliation}

\author{Marcos M. Alexandrino}

%\thanks{ }

\address{Marcos M. Alexandrino\\Instituto de Matem\'{a}tica e Estat\'{\i}stica\\
Universidade de S\~{a}o Paulo, Rua do Mat\~{a}o 1010,05508 090 S\~{a}o Paulo, Brazil}
\email{marcosmalex@yahoo.de}
\email{malex@ime.usp.br}

\thanks{The  author was  supported by CNPq-Brazil}


\date{2009}


\begin{abstract}

Let $\F$ be a singular Riemannian foliation on a compact Riemannian manifold $M$. By   successive blow-ups along the strata of $\F$ we construct a regular Riemannian foliation $\hat{\F}$ on a compact Riemannian manifold $\hat{M}$ and a desingularization map $\hat{\rho}:\hat{M}\rightarrow M$ that projects leaves of $\hat{\F}$ into leaves of $\F$. This result generalizes a previous result due to Molino for the particular case of a singular Riemannian foliation whose leaves   were the closure of leaves of a regular Riemannian foliation.
We also prove that, if the leaves of $\F$ are compact, then, for each small $\epsilon>0,$ we can find $\hat{M}$ and $\hat{\F}$ so that the desingularization map induces  an $\epsilon$-isometry between $M/\F$ and $\hat{M}/\hat{\F}$. This implies in particular that the space of leaves $M/\F$ is a Gromov-Hausdorff limit of a sequence of Riemannian orbifolds $\{(\hat{M}_{n}/\hat{\F}_{n})\}.$

%We also prove that the space of leaves $M/\F$ of a  singular Riemannian foliation $\F$ with compact leaves on a compact manifold $M$ is the limit (in %the sense of Gromov-Hausdorff convergence) of a sequence of space of leaves of  regular Riemannian foliations $\{(\hat{M}_{n},\hat{\F}_{n})\}.$

\end{abstract}
\maketitle

\section{Introduction}

In this section, we  recall some definitions and state our main results as Theorem \ref{thm-Blowup-srf}, Theorem \ref{cor-epsilon-isometria}  and Corollary \ref{cor-gromov-hausdorff-convergence}.

We start by recalling the definition of a singular Riemannian foliation (see  the  book of   Molino \cite{Molino}).
\begin{dfn}[s.r.f]
 A partition $\F$ of a complete Riemannian manifold $M$ by connected immersed submanifolds (the \emph{leaves}) is called a {\it singular Riemannian foliation} (s.r.f for short) if it verifies condition (1) and (2):

\begin{enumerate}
\item $\F$ is a {\it singular foliation},
i.e., the module $\singularF$ of smooth vector fields on $M$ that are tangent at each point to the corresponding leaf acts transitively on each leaf. In other words, for each leaf $L$ and each $v\in TL$ with footpoint $p,$ there is $X\in \singularF$ with $X(p)=v$.
\item   Every geodesic that is perpendicular at one point to a leaf is \emph{horizontal}, i.e.,  is perpendicular to every leaf it meets.
\end{enumerate}
\end{dfn}

Typical examples of s.r.f are the partition by orbits of an isometric action, by leaf closures of a Riemannian foliation (see Molino \cite{Molino}), examples constructed by suspension of homomorphisms (see  \cite{Alex2,Alex4}), examples constructed by changes of metric and surgery (see Alexandrino and T\"{o}ben \cite{AlexToeben}), isoparametric foliations on space forms (some of them with inhomogeneous leaves as in Ferus, Karcher and M\"{u}nzner \cite{FerusKarcherMunzner}) and  partitions by parallel submanifolds of an equifocal submanifold (see Terng and Thorbergsson \cite{TTh1}).

Let $\F$ be a singular Riemannian foliation on a complete Riemannian manifold $M.$  A leaf $L$ of $\F$ (and each point in $L$) is called \emph{regular} if the dimension of $L$ is maximal, otherwise $L$ is called {\it singular}. The union of the leaves having the same dimension is an embedded submanifold called \emph{stratum} and in particular the \emph{minimal stratum}  is a closed submanifold (see Molino \cite{Molino}).

We are now able to state our main result.

\begin{thm}
\label{thm-Blowup-srf}
Let $\F$ be a singular Riemannian foliation of a compact Riemannian manifold $(M,g)$, $\Sigma$ the minimal stratum of $\F$ (with leaves of dimension $k_{0}$) and $\tub_{r}(\Sigma)$ the geometric tube over $\Sigma$ of radius $r$. Then, by blowing up $M$ along $\Sigma,$ we have a singular Riemannian foliation $\hat{\F}_{r}$ (with leaves of dimension greater then $k_{0}$) on a compact Riemannian manifold $(\hat{M}_{r}(\Sigma),\hat{g}_{r})$  and a  map $\hat{\pi}_{r}:\hat{M}_{r}(\Sigma)\rightarrow M$ with the following properties:
\begin{enumerate}
\item[(a)] $\hat{\pi}_{r}$ projects each leaf of $\hat{\F}_{r}$ into a leaf of $\F$. 
\item[(b)]Set $\hat{\Sigma}:=\hat{\pi}_{r}^{-1}(\Sigma)$. Then  $\hat{\pi}_{r}:(\hat{M}_{r}(\Sigma)-\hat{\Sigma},\hat{\F}_{r})\rightarrow (M-\Sigma,\F)$ is a foliated diffeomorphism and $\hat{\pi}_{r}:\hat{M}_{r}(\Sigma)-\tub_{r}(\hat{\Sigma})\rightarrow M-\tub_{r}(\Sigma)$ is an isometry
\item[(c)]If a unit speed geodesic $\hat{\gamma}$ is orthogonal to $\hat{\Sigma}$, then $\hat{\pi}_{r}(\hat{\gamma})$ is a unit speed geodesic orthogonal to $\Sigma$. 
\item[(d)]$\hat{\pi}_{r}|_{\hat{\Sigma}}: (\hat{\Sigma},\hat{g}_{r})\rightarrow(\Sigma,g)$ is a Riemannian submersion. In addition $(\hat{\Sigma},\hat{\F}_{r}|_{\hat{\Sigma}},\hat{g}_{r})$ is a s.r.f and the liftings of horizontal geodesics of $(\Sigma,\F|_{\Sigma},g)$ are horizontal geodesics of $(\hat{\Sigma},\hat{\F}_{r}|_{\hat{\Sigma}},\hat{g}_{r}).$ 

\end{enumerate}
 Furthermore, by successive blow-ups, we have a regular Riemannian foliation $\hat{\F}$ on a compact Riemannian manifold $\hat{M}$ and a desingularization map $\hat{\rho}:\hat{M}\rightarrow M$ that projects each leaf $\hat{L}$ of $\hat{\F}$ into a leaf $L$ of $\F.$  

 \end{thm}

The above theorem generalizes a result due to Molino \cite{MolinoBlowup} who proved items (a) and (b)  under the additional conditions that  the leaves of $\F$ are the closure of leaves of a regular Riemannian foliation.

\begin{rem}
In \cite{Toeben}  T\"{o}ben   used the blow-up technique (on Grassmannian manifold) to study  equifocal submanifolds.
Lytchak \cite{Lytchak2}  generalized the  blow-up introduced by T\"{o}ben  and proved that a singular Riemannian foliations admits a resolution preserving the transverse geometry if and only if it is infinitesimally polar.
\end{rem}

\begin{rem}
\label{rem-Blowup-closed-embedded}

Let $\F$ be a singular Riemannian foliation of a complete Riemannian manifold $(M,g)$. Suppose that the leaves of $\F$ are closed embedded.
 Then the conclusion of Theorem \ref{thm-Blowup-srf} remains valid, if  the tubular neighborhood $\tub_{r}(\Sigma)$ is replaced by a  $\F$-invariant neighborhood $V$ of $\Sigma$ in $M$ and the tubular neighborhood $\tub_{r}(\hat{\Sigma})$ is replaced by the neighborhood $\hat{\pi}^{-1}_{r}(V)$ of $\hat{\Sigma}$ in $\hat{M}_{r}(\Sigma)$.
In particular, by successive blow-ups, we have a regular Riemannian foliation $\hat{\F}$ on a complete Riemannian manifold $\hat{M}$ and a desingularization map $\hat{\rho}:\hat{M}\rightarrow M$ that projects each leaf $\hat{L}$ of $\hat{\F}$ into a leaf $L$ of $\F.$  
This can be proved, following the proof of 
 Theorem \ref{thm-Blowup-srf} and replacing $\tub_{r}(\Sigma)$ by the neighborhood $V$ constructed in Proposition \ref{viz-F-invariante}.
\end{rem}

In the particular case of s.r.f with compacts leaves on a compact manifold, we conclude that,
for each small $\epsilon>0$, we can find $\hat{M}$ and $\hat{\F}$ so that the desingularization map induces  an $\epsilon$-isometry between $M/F$ and $\hat{M}/\hat{\F}$. In other words, we have the next result.

\begin{thm}
\label{cor-epsilon-isometria}
Let $\F$ be a singular Riemannian foliation on a compact Riemannian manifold $M$. Assume that the leaves of $\F$ are compact. Then for each small positive number $\epsilon$ we can find a regular Riemannian foliation $\hat{\F}$ with compact leaves on a compact Riemannian manifold $\hat{M}$ and a desingularization map $\hat{\rho}:\hat{M}\rightarrow M$  with the following property:
if $\hat{q},\hat{p}\in \hat{M}$ then
\[ |d(L_{q},L_{p})-\hat{d}(\hat{L}_{\hat{q}},\hat{L}_{\hat{p}})|<\epsilon\]
where $p=\hat{\rho}(\hat{p})$ and $q=\hat{\rho}(\hat{q})$. 

\end{thm}

The above result implies directly the next corollary (see appendix). 

%that the space of leaves $M/\F$ of a  singular Riemannian foliation $\F$ with compact leaves on a compact manifold $M$ is the limit (in the sense of %Gromov-Hausdorff convergence) of a sequence of space of leaves $\hat{M}_{n}/\hat{\F}_{n}$ of  (regular) Riemannian foliations %$\{(\hat{M}_{n},\hat{\F}_{n})\}.$

\begin{cor}
\label{cor-gromov-hausdorff-convergence}
Let $\F$ be a singular Riemannian foliation on a compact Riemannian manifold $M$. Assume that the leaves of $\F$ are compact. Then for each small positive number $\epsilon$ we can find a regular Riemannian foliation $\hat{\F}$ with compact leaves on a compact Riemannian manifold $\hat{M}$ so that
$d_{G-H}(M/\F, \hat{M}/\hat{\F})<\epsilon$, where $d_{G-H}$ is the distance of Gromov-Hausdorff.
\end{cor}

\begin{rem}
Recall that if $\hat{\F}$ is a Riemannian foliation  with compact leaves on a complete Riemannian manifold $\hat{M}$, then $\hat{M}/\hat{\F}$ is a Riemannian orbifold (see Molino \cite{Molino}).   
Therefore the above corollary implies that $M/\F$ is a Gromov-Hausdorff limit of a sequence of Riemannian orbifolds $\{(\hat{M}_{n}/\hat{\F}_{n})\}.$

\end{rem}

\begin{rem}
Theorem \ref{cor-epsilon-isometria} and Corollary \ref{cor-gromov-hausdorff-convergence} remain valid if we assume that $\F$ is a s.r.f on a complete Riemannian manifold such that the leaves of $\F$ are closed embedded and $M/\F$ is compact. This can be proved using Remark \ref{rem-Blowup-closed-embedded} and following the proof of Theorem \ref{cor-epsilon-isometria}. 

%Note that in this case the neighborhood $V$ of $\Sigma$ can be chosen to be $\tub_{r}(\Sigma)$, because $M/\F$ is compact, and hence it is not %necessarly to change anything in the proof of the corollaries.

\end{rem}

This paper is organized as follows. In Section 2 we discuss some results from the theory of  s.r.f that are used in the proof of Theorem \ref{thm-Blowup-srf}. In Section 3  and 4 we prove Theorem \ref{thm-Blowup-srf} and Theorem \ref{cor-epsilon-isometria} respectively. 
Finally, in Section 5 (appendix) we recall some basic facts about Gromov-Hausdorff distance  that  imply  that Corollary \ref{cor-gromov-hausdorff-convergence} follows  from Theorem \ref{cor-epsilon-isometria}.

%Finally  in Sections 4 and 5 we prove Corollary \ref{cor-epsilon-isometria} and Corollary \ref{cor-gromov-hausdorff-convergence} respectively. 

%%%%%%%%%%%%%%%%%%%%%%%%%%%%%%%%%%%%%%%%%%%%%%%%%%%%%%%%%%%%%%%%%%%%%%%%%%%%%%%%%%%%%%
\section{Properties of a s.r.f.} 
\label{sec-prop}

In this section we  review some results and proofs of \cite{Molino} and \cite{AlexToeben2}
that will be needed to prove  Theorem \ref{thm-Blowup-srf}. 
We also present some new propositions.

We start by recalling  equivalent definitions of regular Riemannian foliations.

\begin{prop}[\cite{Molino}]
\label{prop-equivalencia-df-FR}
Let $\F$ be a  foliation on a complete Riemannian manifold $(M,\metric)$.
Then the following statements are equivalent
\begin{enumerate}
\item[(a)] $\F$ is a Riemannian foliation.
\item[(b)] For each $q\in M$  there exists a neighborhood $U$ of $q$ in $M$, a Riemannian manifold $(\sigma,b)$     and a Riemannian submersion $f: (U,\metric)\rightarrow (\sigma,b)$ such that the connected components of $\F \cap U$ (plaques) are pre images of $f.$
\item[(c)] Set $\metric_{T}:=A^{*}\metric$ where $A_{p}:T_{p}M\rightarrow \nu_{p}L$ is the orthogonal projection. Then the  Lie derivative $L_{X}\metric_{T}$ is zero for each $X\in \singularF ,$ where $\singularF $ is the module of smooth vector fields on $M$ that are tangent at each point to the corresponding leaf. In this case $\metric_{T}$ is called the transverse metric.
\end{enumerate}
\end{prop}

Throughout the rest of this section we assume that $\F$ is a singular Riemannian foliation (s.r.f) on a complete Riemannian manifold $M.$

The first interesting result about s.r.f. is the so called \emph{Homothetic Transformation Lemma} of Molino (see \cite[Lemma 6.2]{Molino}).

%We start by recalling the

By conjugating the homothetic transformations of the normal bundle $\nu P$ of a plaque $P$  via the normal exponential map, one defines for small strictly positive real numbers $\lambda$, a homothetic transformation $h_{\lambda}$  with proportionality constant $\lambda$ with respect to the plaque $P.$

\begin{prop}[\cite{Molino}]
\label{homothetic-lemma}
The homothetic transformation $h_{\lambda}$ sends plaque to plaque and therefore respects the singular foliation $\F$ in the tubular neighborhood $\tub(P)$ where it is defined.
\end{prop}

The next two propositions contain some  improvements of  Molino's results (compare with Theorem 6.1 and Proposition 6.5 of \cite{Molino}).

%For the sake of completeness, we review the main ideias involved.

\begin{prop}[\cite{AlexToeben2}]
\label{lemma-almost-product}
\
Let $g$ be the original metric on $M$ and $q\in M.$  Then there exists a tubular neighborhood $\tub(P_{q})$ and  a new metric $\tilde{g}$ on $\tub(P_{q})$ with the following properties.
\begin{enumerate}
\item[(a)] For each $x\in \tub(P_{q})$ the normal space of the leaf $L_x$ is tangent to the slice $S_{\tilde{q}}$ which contains $x$, where $\tilde q\in P_q.$ 
\item[(b)] Let $\pi:\tub(P_{q})\rightarrow P_q$ be the radial projection. Then the restriction $\pi|_{P_{x}}$ is a Riemannian submersion.
\item[(c)] $\F\cap \tub(P_{q})$ is a s.r.f.
%\item[(d)] $\F\cap S_{\tilde{q}}$ is a s.r.f. for each $\tilde{q}\in P_{q}.$
\item[(d)] The associated transverse metric is not changed, i.e., the distance between the plaques with respect to $g$ is the same distance between the plaques with respect to $\tilde{g}$.
\item[(e)] If a curve $\gamma$ is a geodesic orthogonal to $P_{q}$ with respect to the original metric $g$, then 
$\gamma$ is a geodesic orthogonal to $P_{q}$ with respect to the new metric $\tilde{g}$.
\end{enumerate}
\end{prop}
\begin{proof}

Let $X_1,\ldots, X_r \in  \singularF$ (i.e., vector fields that are always tangent to the leaves) 
so that $\{X_i(q)\}_{i=1,\ldots ,r}$ is a linear basis of $T_{q}P_{q}$.  Let $\varphi_{t_1}^1,\ldots, \varphi_{t_r}^r$ denote the associated one parameter groups and define $\varphi(t_1,\ldots,t_r,y):= \varphi_{t_1}^1\circ\cdots\circ\varphi_{t_r}^r$ where $y\in S_{q}$ and $(t_1,\ldots,t_r)$ belongs to a neighborhood $U$  of $0\in\mathbb{R}^r.$  Then, reducing $U$ and $\tub(P_{q})$ if necessary, one can guarantee the existence of   a regular foliation $\F^{2}$ with plaques   $P^{2}_{y}= \varphi(U,y).$ We note that the plaques $P^{2}_{z}\subset P_{z}$ and each plaque $P^{2}$ cuts each slice  at exactly one point.  Using the fact that  $\pi|_{P^{2}_{y}}: P^{2}_{y}\rightarrow P_{q}$ is a diffeomorphism, we can define a metric on each plaque $P^{2}_{y}$ as
$\tilde{g}^{2}:=(\pi|_{P^{2}_{y}})^{*}g.$

Now we want to define a metric $\tilde{g}^1$ on each slice $S\in \{S_{\tilde{q}}\}_{\tilde{q}\in P_{q}} .$ Set $D_{p}:=\nu_{p} L^{2}_{p}$ and define $\Pi:T_{p}M\rightarrow D_{p}$ as the orthogonal projection with respect to $g$. The fact that each plaque $P^{2}$ cuts each slice  at  one point implies that $\Pi|_{T_{p}S}:T_{p}S\rightarrow D_{p}$ is an isomorphism. Finally we define $\tilde{g}^{1}:= (\Pi|_{T_{p}S})^{*}g$ and $\tilde{g}:=\tilde{g}^{1}+\tilde{g}^{2}$, meaning that $\F^2$ and the slices meet orthogonally. Items (a) and (b) follow directly from the definition of $\tilde{g}.$ 

%%%%%%%%%%%%%%%%%%%%%%%%%%%%%%%%%%% NOVA PARTE
To prove Item (c) it suffices to prove that the plaques of $\F$ are locally equidistant to each other. Let $x\in S_{\tilde{q}}$, $P_{x}$ a plaque of $\F$. We know that the plaques of $\F$ are contained in the leaves of the foliation by distance-cylinders $\{C\}$ with axis $P_{x}$ with respect to $g$.
We will prove that each $C$ is also a distance-cylinder with axis $P_{x}$ with respect to the new metric $\tilde{g}.$ These facts and  the arbitrary choice of $x$  will imply that the plaques of $\F$ are locally equidistant to each other.

First we recall that a smooth function $f:M\rightarrow \mathbb{R}$ is called a \emph{transnormal function} with respect to the metric $g$ if  there exists a $C^{2}(f(M))$ function  $b$  such that $g(\grad f,\grad f)=b\circ f$.
According to Q-M Wang \cite{Wang} \emph{there are at most two critical level sets of the transnormal function $f$ and each regular level set of $f$ is a distance cylinder over them}. 
 
Let $f:\tub(P_{x})\rightarrow \mathbb{R}$ be a smooth transnormal function with respect to the metric $g$ so that each regular level set $f^{-1}(c)$ is a cylinder $C$ with axis $P_x$,  e.g. $f(y)=d(y,P_x)^2$. 
Let $\widetilde{\grad} f$ denote the gradient of $f$ with respect to the metric $\tilde{g}.$ It follows from the construction of $\tilde{g}$ that 
 \begin{equation}
 \label{eq-widetildegrad-grad}
 \widetilde{\grad} f=\grad f+l 
 \end{equation}
 where $l$ is a vector tangent to a plaque of $\F^{2}$ and in particular to a plaque of $\F$. 

Indeed, let $v\in D_p$ and $w:=(\Pi|_{T_{p}S})^{-1}(v)$. Then 
\begin{eqnarray*}
  g(\grad f,v) &=& d f (v) \\
               &=& d f (w)\\
               &=& \tilde{g}(\widetilde{\grad} f, w)\\
	          &=& \tilde{g}^{1}(\widetilde{\grad} f, w)\\
               &=& g(\Pi\widetilde{\grad} f,\Pi w)\\
               &=&g(\Pi \widetilde{\grad}f, v) 
\end{eqnarray*}
We conclude from the arbitrary choice of  $v\in D_p$, that  $\grad f=\Pi \widetilde{\grad} f,$  and hence  
$ \widetilde{\grad} f=\grad f+l$.

Equation (\ref{eq-widetildegrad-grad}) implies that $f$ is a also a transnormal function with respect to the metric $\tilde{g}$, i.e., 
\begin{equation}
\label{f-transnormal-g-tildeg}
\tilde{g}(\widetilde{\grad} f,\widetilde{\grad} f)=b\circ f, 
\end{equation}
Indeed, 
\begin{eqnarray*}
\tilde{g}(\widetilde{\grad} f,\widetilde{\grad} f)&=& d f(\widetilde{\grad} f)\\
                                                  &=& d f(\grad f)\\
                                                  &=& g(\grad f, \grad f)\\
                                                  &=& b\circ f
\end{eqnarray*}

Using a local version of Q-M Wang's theorem \cite{Wang}, we conclude that each regular level set of $f$ (i.e. $C$ ) is a distance cylinder around $P_{x}$ with respect to the metric $\tilde{g}$.

%Now we want to prove Item (d).  Set $P_{x}^{s}=P_{x}\cap S_{\tilde{q}}$ and  $C^{s}:=C\cap S_{\tilde{q}}.$ It suffices to note that the singular %foliation $\{C^s\}$ is a foliation by cylinders with axis $P_{x}^s$ with respect to the new metric $\tilde{g}.$ This follows from the fact that %$\nu_{x}P_{x}\subset T_{x}S_{\tilde{q}}$ and that each geodesic orthogonal to $P_{x}$ at $x$ is contained in $S_{\tilde{q}}$ (see Item (a)).

%In particular we conclude that the distance between $C$ and $P_x$ and the distance between $C^s$ and $P_{x}^{s}$ with respect to the metric %$\tilde{g}$ are the same. 

To prove item (d) we have to prove that the distance between the cylinder $C$ and the plaque $P_{x}$ is the same for both metrics.
Let $f$ be the transnormal function (with respect to $g$) defined above. 
According to Q-M Wang \cite{Wang}  for $k=f(P_x)$ and a regular value $c$  we have  
$d(P_{x},f^{-1}(c))=\int_{c}^{k}\frac{ds}{\sqrt{b(s)}}.$ Since $f$ is also a transnormal function with respect to $\tilde{g}$ (see equation (\ref{f-transnormal-g-tildeg})), we conclude that $d(P_{x},C)=\tilde{d}(P_{x},C),$
for  $C=f^{-1}(c).$

%%%%%%%%%%%%%%%%%%%%%%
Finally we prove item (e). We consider the transnormal function $f$ above with $x=q$. 
In this case,  
equation (\ref{eq-widetildegrad-grad}) and the fact that $\grad f\in D_p\cap T_pS$ imply that $\grad f=\widetilde\grad f$. On the other hand, the integral curves of the gradient of a transnormal function are geodesic segments up to reparametrization (see e.g. \cite{Wang}). Therefore the radial geodesics of $P_{q}$ coincide in both metrics. This finishes the proof.

%%%%%%%%%%%%%%%%%%%%%%%%%%%%%%%%%%%%%%%%%%%%%%%%%%%%%%%%%%%%%

\end{proof}

\begin{prop}[\cite{AlexToeben2}]
\label{lemma-metric-in-S}
There exists  a new metric $g_0$  on $\tub(P_q)$ with the following properties:
\begin{enumerate}
\item[(a)] Consider the tangent space $T_{\tilde{q}}S_{\tilde{q}}$ with the metric $g$ and $S_{\tilde{q}}$ with the metric $g_0$. Then  $\exp_{\tilde{q}}:T_{\tilde{q}}S_{\tilde{q}}\rightarrow S_{\tilde{q}}$ is an isometry, for $\tilde{q}\in P_{q}.$
\item[(b)] For this new metric $g_0$ we have that $\F\cap S_{\tilde{q}}$ and $\F$ restricted to $\tub(P_q)$ are also s.r.f. 
\item[(c)]  For each $x\in \tub(P_{q})$ the normal space of the leaf $L_x$ is tangent to the slice $S_{\tilde{q}}$ which contains $x$, where $\tilde q\in P_q.$
\end{enumerate}

\end{prop}

\begin{rem}
\label{rem-prop-flat-metric-in-S}
Clearly a curve $\gamma$ which is a geodesic orthogonal to $P_{q}$ with respect to the original metric, remains a  geodesic orthogonal to $P_{q}$ with respect to the new metric $g_0$.
\end{rem}

%%%%%%%%%%%%%%%%%%%%%%%%%%%%%%%%%%%%%%%%%%%%%%  novo 16/06

Now we briefly recall some properties about the stratification of $M$.

\begin{prop}[\cite{Molino}]
Each stratum is an embedded submanifold and a union of geodesics that are perpendicular to the leaves. 
\end{prop}

With this result one can easily see that the induced foliation on each stratum is a Riemannian foliation and that the restriction of the metric to the stratum is a bundle-like metric.
The observation that every geodesic perpendicular to a stratum is perpendicular to the leaves allow us to adapt the argument of Proposition \ref{homothetic-lemma} and conclude the next result.

\begin{prop}[\cite{Molino}]
\label{homothetic-lemma-stratum}
Let $\Sigma$ be a singular stratum. Consider a tubular neighborhood $\tub(U)$ of a relatively compact open set $U$ of $\Sigma$. Then the plaques of $\F\cap \tub(U)$ are contained in the  the cylinders of axis $U$. In addition, the foliation is locally invariant by the homothetic transformations with respect to the stratum. 
\end{prop}

An important property of a s.r.f $\F$ is the so called equifocality of $\F.$ Roughly speaking this means that the ``parallel sets" of each leaf  are leaves of $\F$. In order to make this concept precise, we need to recall the definition of foliated vector field. Let $L$ be a leaf of $\F$, $\Sigma_{L}$  the stratum of $L$ and $\singularFsigma$ the module of smooth vector fields on $\Sigma_{L}$ that are tangent at each point to the corresponding leaf.  A vector field $\xi$ on $\Sigma_{L}$ is called \emph{foliated} if for each vector field $Y\in\singularFsigma$ the Lie bracket $[\xi,Y]$ also belongs to $\singularFsigma$.
If we consider a local submersion $f$ which describes the plaques of $\F|_{\Sigma_{L}}$ in a neighborhood of a point of $L$, then a normal foliated vector field is a normal projectable/basic vector field with respect to $f.$

\begin{rem}
The Bott or basic connection $\nabla$ of the foliation $\F|_{\Sigma_{L}}$ is a connection of $T\Sigma$ with $\nabla_XY=[X,Y]^{\nu\F}$ whenever $X\in \singularFsigma$ and $Y$ is vector field of the normal bundle $\nu\F$ of the foliation. Here the superscript $\nu\F$ denotes projection onto $\nu\F$. A foliated vector field clearly is parallel with respect to the Bott connection. This connection can be restricted to the normal bundle of a leaf.

%In the particular case of   s.r.f that admits sections  a normal foliated vector field is a parallel normal field along each regular leaf $L$ with %respect to the induced Levi-Civita connection on $\nu L$ and vice versa.
\end{rem} 

\begin{thm}[\cite{AlexToeben2} Equifocality of $\F$]
\label{thm-s.r.f.-equifocal}
Let $\F$ be a s.r.f. on a complete Riemannian manifold $M$. Then for each  point $p$ there exists a neighborhood $U$ of $p$ in $L_{p}$ such that
\begin{enumerate}
\item[(1)] For each   normal foliated vector field $\xi$  along $U$  the derivative of  the map    $\eta_{\xi}:U\rightarrow M,$ defined as $\eta_{\xi}(x):=\exp_{x}(\xi),$ has constant rank.
\item[(2)] $W:=\eta_{\xi}(U)$ is an open set of $L_{\eta_{\xi}(p)}$.
\end{enumerate}
\end{thm}

\begin{cor}[\cite{AlexToeben2}]
\label{cor}
Let $L_{p}$ be a regular  leaf with trivial holonomy and $\Xi$ denote the set of all  normal foliated vector  fields along $L_{p}.$ 
\begin{enumerate}
\item[(1)] Let $\xi\in \Xi$. Then $\eta_{\xi}:L_{p}\rightarrow L_{q}$ is a covering map if $q=\eta_{\xi}(p)$ is a regular point.
\item[(2)]  $\F=\{\eta_{\xi}(L_{p})\}_{\xi\in \, \Xi},$ i.e., we can reconstruct the singular foliation by taking all parallel submanifolds of the regular leaf $L_{p}.$ 
\end{enumerate}
\end{cor}

\begin{cor}[\cite{AlexToeben2}]
\label{singular-points-isolated}
Let $\gamma$ be a geodesic orthogonal to a leaf $L$ and tangent to the stratum $\Sigma_{L}$.  Then the points of $\gamma$ that do not belong to $\Sigma_{L}$  are isolated on $\gamma.$ 
\end{cor}

\begin{rem}
In \cite{AlexToeben2} we only proved the equifocality of regular leaves. Nevertheless, as can be easily checked,  the  same prove is valid for singular leaves, if one consider  foliated vector fields tangent to the stratum. 
\end{rem}

%%%%%%%%%%%%% fim da mudanca 16/06
 
The equifocality of $\F$  and Proposition \ref{homothetic-lemma-stratum} imply  the next result.

\begin{prop}
\label{cor-transnormal-system}
Let $\F$ be a s.r.f, $\Sigma$ the minimal stratum of $\F$ and $q\in \Sigma$. For each $x\in P_{q},$ set $T_{x}:=\exp_{x}(\nu (\Sigma)\cap B_{\epsilon}(0))$. 
 Then  there exists a neighborhood $U$ of $q$ in $\Sigma$ such that if a plaque $P$ of $\F$ meets $T_{q}$, we have 
 $\emptyset\neq P\cap S_{x}\subset T_{x}$ for all $x\in U\cap P_{q}.$ 
  \end{prop}
  \begin{proof}
  Let $\beta$ be a curve in $P$ such that $\beta(0)\in T_{q}$ and $\beta(1)\in S_{x}$. Let $\gamma:[0,1]\rightarrow T_{q}$ be the minimal geodesic such that  $q=\gamma(1)$ to $\gamma(0)=\beta(0)$. It follows from Proposition \ref{homothetic-lemma} that $\gamma|_{[0,1)}$ is contained in the same stratum of $\beta(0)$. Therefore we can transport the horizontal geodesic $\gamma$ with respect to the Bott connection along $\beta$. Let $||_{\beta}\gamma$ be this transported geodesic and note that $||_{\beta}\gamma(t)\in L_{\gamma(t)}$ for $t\in[0,1]$.

First we will prove that $||_{\beta}\gamma$ is orthogonal to $\Sigma$. Assume that this is not the case. Since $x$ is near to $q$ we have that 
$||_{\beta}\gamma$ is not tangent to $\Sigma$. Then, since $||_{\beta}\gamma$ is not orthogonal to $\Sigma$, there exists a piecewise horizontal geodesic  $\alpha$ that joins $\Sigma$ to $\beta(1)$ so that $\alpha$ coincides with $||_{\beta}\gamma|_{[0,\delta]}$ and $l(\alpha)<l(||_{\beta}\gamma)$, i.e., the length of $\alpha$ is lower then the length of $||_{\beta}\gamma$. Note that $\alpha$ also belongs to the stratum that contains $\beta(1)$ and hence we can transport $\alpha$ with respect to the Bott connection along $\beta^{-1}$. 
We conclude that $l(||_{\beta^{-1}}\alpha)<l(\gamma)$, which contradicts the fact that $\gamma$ is orthogonal to $\Sigma$. Therefore $||_{\beta}\gamma$ is orthogonal to $\Sigma$.

On the other hand, by the equifocality of $\F$, we have  that $||_{\beta}\gamma(1)=x$.
 
These two facts together imply that $||_{\beta}\gamma\subset T_{x}$ and in particular that $\beta(1)\in T_{x}$.

  \end{proof}

We also need the next result due to Molino \cite[Proposition 6.4]{Molino}, which we reformulate as follows.

\begin{prop}
\label{prop-metric-stratum}
Let $\F$ be a s.r.f. with respect to a metric $g$. Let $\tilde{g}$ be another metric with following property: $(\Sigma,\F|_{\Sigma})$ is a Riemannian foliation with respect to $\tilde{g}$, for each stratum $\Sigma$. Then $\F$ is a s.r.f with respect to $\tilde{g}.$  
\end{prop}
\begin{proof}

It suffices to prove the result in a neighborhood of a plaque $P_{q}\subset \Sigma_{r}$. Note that the result is already true for the regular stratum. By induction, we can assume that it is also true for $\cup _{s>r}\Sigma_{s}$.

Let $\tilde{\gamma}$ be a segment of geodesic (with respect to $\tilde{g}$) that is orthogonal to the plaque $P_{q}$ at $\tilde{\gamma}(0)=q$. 

First we consider the case when $\tilde{\gamma}'(0)$ is not tangent to $\Sigma_{r}.$

%%%%%%%%%%%%%%%%%%%%%
\claim 1: \emph{Set $U=\tub(P_{q})\cap (\cup _{s>r}\Sigma_{s}).$  Then $\tilde{\gamma}\cap U$ is orthogonal to the leaves that it meets.}

In order to prove Claim 1 set $\tilde{v}_{0}:=\tilde{\gamma}'(0).$
Let $V_{0}$ be the hyperplane orthogonal to $\tilde{v}_{0}$ relatively to $\tilde{g}.$ Note that $T_{q}P\subset V_{0}$. Let $v_{0}$ be a vector orthogonal to $V_{0}$ relatively to $g$ and $\gamma_{0}$ the geodesic (relatively to $g$) with $\gamma'(0)=v_{0}$.  

At first assume that $v_{0}$ is not tangent to $\Sigma_{r}.$ 

Set $V_{t}:=T_{\gamma(t)}\partial \tub (P_{q})$ where $\partial \tub(P_{q})$ is a geometric cylinder of axis $P_{\gamma(q)}$ with respect to the original metric $g$. Since $\F$ is a s.r.f (with respect to $g$) we have $P_{\gamma(t)}\subset V_{t}$. Now, for each $t$ consider a segment of geodesic $\tilde{\gamma}_{t}$ (with respect to $\tilde{g}$) such that $\tilde{\gamma}_{t}(0)=\gamma(t)$ and $\tilde{\gamma}_{t}'(0)$ is orthogonal to $V_{t}$ relatively to $\tilde{g}$. Then each geodesic $\tilde{\gamma}_{t}$ is orthogonal to the leaves of $\F$ relatively to $\tilde{g}$, because 
 $\tilde{\gamma}_{t}\subset U,$  $P_{\gamma(t)}\subset V_{t}$ and $\F\cap U$ is a s.r.f. Since $\tilde{\gamma}_{t}$ converges to $\tilde{\gamma}$, we conclude that $\tilde{\gamma}$ is orthogonal to the leaves of $\F$ (relatively to $\tilde{g}$). 

Now assume that $v_{0}$ is tangent to $\Sigma_{r}.$ 

Consider a sequence of vectors $\{v_{i}\}$ that converges to $v_{0}$ so that each $v_{i}$ is orthogonal to $P_{q}$ (relatively to $g$) and $v_{i}$ is not tangent to $\Sigma_{r}$.  Define $V^{i}$ the hyperplane orthogonal to $v_{i}$ (relatively to $g$). Since  $V^{i}$ converges to $V_{0}$ we can define a sequence of vectors $\{\tilde{v}_{i}\}$  orthogonal to $V^{i}$ (relatively to $\tilde{g}$) that  converges to $\tilde{v}_{0}$. Let $\tilde{\gamma}_{i}$ be the geodesic (with respect to $\tilde{g}$) such that $\tilde{\gamma}_{i}'(0)=\tilde{v}_{i}$. It follows from the above discussion that each geodesic $\tilde{\gamma}_{i}$ is orthogonal to the leaves of  $\F$ (with respect to $\tilde{g}$). Since $\{\tilde{\gamma}_{i}\}$ converges to $\tilde{\gamma}$ we infer that $\tilde{\gamma}$ is orthogonal to the leaves of $\F$ (with respect to $\tilde{g}$) and this conclude the proof of Claim 1.

\claim 2: \emph{ The segment of geodesic $\tilde{\gamma}$ meets $\Sigma_{r}$ only at $\tilde{\gamma}(0).$}

 %The plaques that meet $\gamma \cap U$ are homothetic to each other by the homothetic transformation with respect to $P_{q}.$

In order to prove Claim 2 assume that $\tilde{\gamma}(t)\in U$ for $0<t<1$.  
First note that $P_{\tilde{\gamma}(t)}$ is contained in an open set of a geometric cylinder of axis $P_{q}.$ One can see  this using the fact that Claim 1 is valid for each geodesic perpendicular to $P_{q}$ and transversal to $\Sigma_{r}$. Then note that $P_{\tilde{\gamma}(t)}$ is contained in an open set of a geometric cylinder of axis $P_{\tilde{\gamma}(1)}.$ This follows from the equifocality of $\F\cap U$ and from the fact that $P_{\tilde{\gamma}(t)}$ converges to $P_{\tilde{\gamma}(1)}.$ Now by the same argument used in the proof of Proposition \ref{homothetic-lemma} we conclude that $P_{\tilde{\gamma}(1)}$ is homothetic to $P_{\tilde{\gamma}(t)}$ and hence has the same dimension of $P_{\tilde{\gamma}(t)}$. Therefore $P_{\tilde{\gamma}(1)}$ is not contained in $\Sigma_{r}$ and in particular $\tilde{\gamma}(1)$ does not belong to $\Sigma_{r}$.

%Claim 1 is valid for each geodesic perpendicular to $P_{q}$ and transversal to $\Sigma_{r}$. This implies that $P_{\gamma(t)}$ is contained in an open %set of a geometric cylinder of axis $P_{q}.$ 
% Since $\gamma(t)$ converges to $\gamma(1)$ and $\F$ is a s.r.f with respect to the original we conclude that $P_{\gamma(t)}$ converges to %$P_{\gamma(1)}$. Using this fact and the equifocality of $\F\cap U$ we can conclude that $P_{\gamma(t)}$ is contained in an open set of a geometric %cylinder of axis $P_{\gamma(1)}.$ Now we can use the same argument of the proof of  

%This fact and the fact that $\F\cap U$ is a s.r.f allow us to use the same argument used in the proof of Proposition \ref{homothetic-lemma} to %conclude the proof of Claim 2. 

Finally we consider the case when $\tilde{\gamma}'(0)$ is tangent to $\Sigma_{r}.$

\claim 3:\emph{ If $\tilde{\gamma}$ is a geodesic orthogonal to $P_{q}$ and tangent to $\Sigma_{r}$, then $\tilde{\gamma}\subset \Sigma_{r}.$ In particular, since $(\F,\Sigma_{r})$ is a Riemannian foliation, $\tilde{\gamma}$ is orthogonal to the plaques that it meets.}

Set $B=\Sigma_{r}\cap\exp_{q}(\nu P_{q})\cap \tub(P_{q})$. In order to prove  Claim 3 we will prove that $B=\exp_{q}(T_{q}\Sigma_{r}\cap\nu_{q} P_{q})\cap\tub(P_{q})$. Assume that $B$ is not contained in 
$\tub(P_{q}) \cap \exp_{q}(T_{q}\Sigma_{r}\cap\nu_{q} P_{q})$. Then  there exists a point $x\in\Sigma_{r}$ and a geodesic $\tilde{\gamma}$ orthogonal to $P_{q}$ joining $x$ to $q$ such that $\tilde{\gamma}$ is not tangent to $\Sigma_{r}.$
 This contradicts Claim 2. Therefore the submanifold $B$ is an open set of the submanifold 
$\exp_{q}(T_{q}\Sigma_{r}\cap\nu_{q} P_{q})\cap\tub(P_{q})$. It is easy to check that $B$ is closed in  $\exp_{q}(T_{q}\Sigma_{r}\cap\nu_{q} P_{q})\cap\tub(P_{q})$ and hence  $B= \exp_{q}(T_{q}\Sigma_{r}\cap\nu_{q} P_{q})\cap\tub(P_{q}).$

\end{proof}

One can adapt the proof of the above proposition to get the next technical result. 
\begin{prop}
\label{rem-prop-metric-stratum}
Let $\F$ be a singular foliation on a Riemannian manifold $(M,g)$ with the following properties:
\begin{enumerate}
\item[(a)] Each stratum is a submanifold.
\item[(b)] The sequence of plaques $\{P_{x_{n}}\}$ converges to $P_{x}$ when the sequence $\{x_{n}\}$ converges to $x$.
\item[(c)] The restriction of $\F$ to each stratum is a Riemannian foliation. 
\item[(d)] For each vector $v$ that is not tangent to a stratum $\Sigma$ there exists a curve $\beta:[0,\epsilon)\rightarrow M-\Sigma$ and a distribution $t\rightarrow V_{t}$ along $\beta$  such that
 \begin{enumerate}
 \item[(d.1)]$T_{\beta(t)}P_{\beta(t)}\subset V_{t},$ 
  \item[(d.2)]$v$ is orthogonal to $V_{0}.$ 
\end{enumerate}
\end{enumerate}
Then $\F$ is a s.r.f with respect to $g$.

%\textbf{
%In order to formulate this as a proposition, one should discuss the  definition of convergence of   plaques of a %general singular foliation. Since we will use this remark only for a very particular singular foliation, where this %concept will be clear,  we  avoid this kind of discussion here.  
%}
\end{prop}

We conclude this section constructing a metric $g^{0}$ on a neighborhood of $\Sigma$ with properties similar to the properties of the metric defined in Proposition \ref{lemma-metric-in-S}.

\begin{lemma}
\label{prop-flat-metric}
Let $\F$ be a s.r.f on a complete Riemannian manifold $M$ with metric $g$ and 
$\Sigma$ the minimal stratum (with leaves of dimension $k_{0}$).
Then, for each $q\in\Sigma$ there exists a neighborhood $U_{\alpha}$ of $q$ in $\Sigma$
and metric  $g_{\alpha}^{0}$ on  $\tub_{r}( U_{\alpha})$ such that:
\begin{enumerate}
\item[(a)] $\F$ restrict to  $\tub_{r}(U_{\alpha})$ is a Riemannian foliation with respect to $g_{\alpha}^{0}$.
%\item[(b)] If a curve $\gamma$ is a geodesic orthogonal to $ U_{\alpha}$ with respect to the original metric, then 
%$\gamma$ is a geodesic orthogonal to $ U_{\alpha}$ with respect to the new metric $g_{\alpha}^{0}$.
\item[(b)] There exists a smooth distribution $H$, whose dimension is equal to the codimension of $L_{q}$, such that the normal space  of each plaque of $\F|_{\tub( U_{\alpha})}$ (with respect to $g_{\alpha}^{0}$) is contained in $H.$
\item[(c)] For $x\in U_{\alpha}$, consider the tangent space $ T_{x}S_{x}$ with the metric $g$. Then $d (\exp_{x})_{\xi}:(T_{\xi}T_{x}S_{x})\rightarrow H_{\exp_{x}(\xi)}$ is an isometry, where  $\xi\in\nu_{x} \Sigma.$
\end{enumerate}
\end{lemma}
\begin{proof}
For $\xi\in\nu_{x}(\Sigma)$  set $H_{\exp_{x}(\xi)}:=T_{\exp_{x}(\xi)}S_{x}$. 
Let $\F^{2}$ be the foliation constructed in the proof of Proposition \ref{lemma-almost-product}. Reducing  $\tub_{r}( U_{\alpha})$ if necessary, note that the plaques $P^{2}_{z}\subset P_{z}$ and each plaque $P^{2}$ cuts each slice  at exactly one point. In particular  
$TP^{2}_{\exp_{x}(\xi)}\oplus H_{\exp_{x}(\xi)}=T_{\exp_{x}(\xi)}M.$ 
On each $TP^{2}_{\exp_{x}(\xi)}$  consider the metric 
$(g^{0}_{\alpha})^{2}:=(\pi|_{TP^{2}_{\exp_{x}(\xi)}})^{*}g$, where $\pi:\tub_{r}( U_{\alpha})\rightarrow \Sigma$ is the radial projection.  Define a metric $(g^{0}_{\alpha})^{1}$ on $H_{\exp_{x}(\xi)}$ such that $\exp_{x}(T_{\xi}T_{x}S_{x})\rightarrow H_{\exp_{x}(\xi)}$ is an isometry. Finally set
$g^{0}_{\alpha}:=(g^{0}_{\alpha})^{1}+(g^{0}_{\alpha})^{2}$, meaning that $\F^{2}$ and the distribution $H$ meet orthogonally.

Let $P$ be a plaque and $\Sigma_{P}$ the stratum that contains $P$. It follows from Proposition \ref{cor-transnormal-system} that the transverse metric restrict to $\Sigma_{P}$ on $P$ coincides with the transverse metric constructed in  Proposition \ref{lemma-metric-in-S}. By Proposition \ref{lemma-metric-in-S} the Lie derivative $L_{X}(g^{0}_{\alpha})_{T}$ is zero. We can use the same argument to conclude that 
the Lie derivate of the transverse metric $(g^{0}_{\alpha})_{T}$ along each other plaque of $\Sigma_{P}$ is also zero and 
 hence, by Proposition \ref{prop-equivalencia-df-FR}, $\F|_{\Sigma}$ is a Riemannian foliation with respect to $g^{0}_{\alpha}$. Therefore, by Proposition \ref{prop-metric-stratum}, $\F$ is a singular Riemannian foliation with respect to $g^{0}_{\alpha}.$

\end{proof}

\begin{prop}
\label{prop-flat-metric-N}
Let $\F$ be a s.r.f on a compact Riemannian manifold $(M,g)$ and  $\Sigma$  the minimal stratum (with leaves of dimension $k_{0}$). 
Then there exists  a neighborhood  $\tub_{r}(\Sigma)$ of $\Sigma$   and a metric  $g^{0}$ on $\tub_{r}(\Sigma)$   such that:
\begin{enumerate}
\item[(a)] $\F$ restrict to  $\tub_{r}(\Sigma)$ is a Riemannian foliation with respect to $g^{0}$.
%\item[(b)] If a curve $\gamma$ is a segment of geodesic orthogonal to $\Sigma$ with respect to the original metric $g$, then 
%$\gamma$ is a geodesic orthogonal to $\Sigma$ with respect to the new metric $g^{0}$.
\item[(b)] There exists a smooth distribution $H$, whose dimension is equal to the codimension of the leaves in $\Sigma$, such that the normal space  of each plaque of $\F|_{\tub_{r}(\Sigma)}$ (with respect to $g^{0}$) is contained in $H.$
\item[(c)] For $q\in \Sigma$, consider the tangent space $ T_{q}S_{q}$ with the metric $g$. Then $d (\exp_{q})_{\xi}:(T_{\xi}T_{q}S_{q})\rightarrow H_{\exp_{q}(\xi)}$ is an isometry, where  $\xi\in\nu_{q} \Sigma.$
%\item[(d)] The restriction of the metric $g^{0}$ to the stratum $\Sigma$ coincides to the restriction of the original metric $g$ to $\Sigma.$ 
\end{enumerate}
\end{prop}
\begin{proof}

Consider a open covering $\{U_{\alpha}\}$  of $\Sigma$ by neighborhoods defined in Lemma \ref{prop-flat-metric} and a partition of unity $\phi_{\alpha}$ subordinate to it.
Set $g^{0}:=\sum \phi_{\alpha} g_{\alpha}^{0}.$

To prove that $\F$ is a s.r.f  it suffices to prove that the plaques of $\F$ are locally equidistant to each other. Let $x\in S_{q}$, $P_{x}$ a plaque of $\F$. For a fixed metric $g_{\alpha}^{0}$, we know that the plaques of $\F$ are contained in the leaves of the foliation by distance-cylinders $\{C\}$ with axis $P_{x}$ with respect to $g_{\alpha}^{0}$.
We will prove that each $C$ is also a distance-cylinder with axis $P_{x}$ with respect to the new metric $g^{0}.$ These facts and  the arbitrary choice of $x$  will imply that the plaques of $\F$ are locally equidistant to each other.

As  we have recalled before,  a smooth function $f:M\rightarrow \mathbb{R}$ is called a \emph{transnormal function} with respect to the metric $g$ if  there exists a $C^{2}(f(M))$ function  $b$  such that $g(\grad f,\grad f)=b\circ f$. 
According to Q-M Wang \cite{Wang} \emph{there are at most two critical level sets of the transnormal function $f$ and each regular level set of $f$ is a distance cylinder over them}. 

Let $f:\tub(P_{x})\rightarrow \mathbb{R}$ be a smooth transnormal function with respect to the metric $g_{\alpha}^{0}$, so that each regular level set $f^{-1}(c)$ is a cylinder $C$ with axis $P_x$,  e.g. $f(y)=d(y,P_x)^2$.

\claim 1 : \begin{enumerate}
\item[(1)]$C$ is a a distance-cylinder with axis $P_{x}$ with respect to each metric $g_{\beta}^{0}$. 
\item[(2)] $(\grad f)_{\alpha}^{0}=(\grad f)_{\beta}^{0},$ where  $(\grad f)_{\alpha}^{0}$ and  $(\grad f)_{\beta}^{0}$ are the gradients with respect to the metrics $g_{\alpha}^{0}$ and $g_{\beta}^{0}$.
\item[(3)] $b:=b_{\alpha}^{0}=b_{\beta}^{0}.$
\end{enumerate}
  In fact, $(\grad f)_{\alpha}^{0}$ and $(\grad f)_{\beta}^{0}$ are orthogonal to $C$ with respect to $g_{\alpha}^{0}$ and $g_{\beta}^{0}$. In particular they are perpendicular to $P_{y}$. We conclude that $(\grad f)_{\alpha}^{0}$ and $(\grad f)_{\beta}^{0}$ belong to $H_{y}$ (see item (b) of  Lemma \ref{prop-flat-metric}). Since $g_{\alpha}^{0}|_{H}=g_{\beta}^{0}|_{H}$ (see item (c) of Lemma \ref{prop-flat-metric}), we conclude that $(\grad f)_{\alpha}^{0}=(\grad f)_{\beta}^{0}$. This implies in particular that $b_{\alpha}^{0}=b_{\beta}^{0}.$

\claim 2 : \begin{enumerate}
\item[(1)] $f$ is a transnormal function with respect to $g^{0}.$
\item[(2)] $b^{0}=b$.
\end{enumerate}
In order to prove Claim 2, first note that $(\grad f)^{0}=(\grad f)_{\alpha}^{0}$. In fact

\begin{eqnarray*}
g^{0}((\grad f)_{\alpha}^{0},V)&=& \sum_{\beta} \phi_{\beta} g_{\beta}^{0}((\grad f)_{\alpha}^{0},V)\\
                             &=& \sum_{\beta} \phi_{\beta} g_{\beta}^{0}((\grad f)_{\beta}^{0},V)\\
                             &=& \sum_{\beta} \phi_{\beta} d f (V)\\
                             &=& d f (V)\\
                             &=& g^{0}((\grad f)^{0},V).
\end{eqnarray*}
 
Now we can see that $f$ is a  transnormal function with respect to $g^{0}$.

\begin{eqnarray*}
g^{0}((\grad f)^{0},(\grad f)^{0})&=&\sum_{\beta} \phi_{\beta} g_{\beta}^{0}((\grad f)_{\beta}^{0},(\grad f)_{\beta}^{0})\\
                              &=& \sum_{\beta} \phi_{\beta} b \circ f\\
                              &=& b \circ f.
                              \end{eqnarray*}

Using a local version of Q-M Wang's theorem, we conclude that each regular level set of $f$ (i.e., $C$ ) is a distance cylinder around $P_{x}$ with respect to the metric ${g}^{0}$.
 
%To conclude the proof we need to change the metric $\tilde{g}^{0}$ so that it also satisfies Item (e). We will do this changing only the metric of the %orthogonal distribution $H^{\perp}$ to $H$
%so that for each $p\in\Sigma$ the metric $H_{p}^{\perp}$ coincides with $T_{p}P$. 
% Since $H^{\perp}_{x}\subset T_{x}P_{x}$, like in Proposition  \ref{prop-flat-metric}, we can use Proposition \ref{cor-transnormal-system},  %Proposition \ref{prop-equivalencia-df-FR}  and   Proposition \ref{prop-metric-stratum} to conclude that 
% this kind of change of metric will preserve the other desired properties. 

\end{proof}

The conclusion of the above result remains valid if $M$ is a complete Riemannian manifold and the leaves of $\F$ are closed embedded.

\begin{prop}
\label{viz-F-invariante}
Let $\F$ be a s.r.f on a complete Riemannian manifold $(M,g)$ and  $\Sigma$  the minimal stratum (with leaves of dimension $k_{0}$). 
Assume that the leaves of $\F$ are closed embedded.
Then there exists  a $\F$-invariant neighborhood  $V$ of $\Sigma$    and a metric  $g^{0}$ on $V$   such that:
\begin{enumerate}
\item[(a)] $\F$ restrict to  $V$ is a Riemannian foliation with respect to $g^{0}$.
%\item[(b)] If a curve $\gamma$ is a segment of geodesic orthogonal to $\Sigma$ with respect to the original metric $g$, then 
%$\gamma$ is a geodesic orthogonal to $\Sigma$ with respect to the new metric $g^{0}$.
\item[(b)] There exists a smooth distribution $H$, whose dimension is equal to the codimension of the leaves in $\Sigma$, such that the normal space  of each plaque of $\F|_{V}$ (with respect to $g^{0}$) is contained in $H.$
\item[(c)] For $q\in \Sigma$, consider the tangent space $ T_{q}S_{q}$ with the metric $g$. Then $d (\exp_{q})_{\xi}:(T_{\xi}T_{q}S_{q})\rightarrow H_{\exp_{q}(\xi)}$ is an isometry, where  $\xi\in\nu_{q} \Sigma.$
%\item[(d)] The restriction of the metric $g^{0}$ to the stratum $\Sigma$ coincides to the restriction of the original metric $g$ to $\Sigma.$ 
\end{enumerate}
\end{prop}
\begin{proof}
We must find a $\F$-invariant neighborhood $V$ of $\Sigma$ in $M$ such that $\exp$ is a diffeomorphism between a neighborhood of  $\Sigma$ in  $\nu\Sigma$ and the neighborhood $V$. Then the rest of  proof is analogous to the proof of Proposition \ref{prop-flat-metric-N}  if one replaces $\tub_{r}(U_{\alpha})$ by  $\tub_{r}(U_{\alpha})\cap V.$
In what follows we construct the neighborhood $V$.

\claim 1: \emph{For each $L$ there exists a small $r$ such that $\tub_{r}(L)$ is a geometric tube}.

In order to prove Claim 1, consider a point $x_{0}\in L$. Since $L$ is closed embedded we can find $r>0$ such that
\begin{itemize}
\item $B_{2r}(x_{0})\cap L$ is a connected submanifold.
\item For $x\in (B_{2r}(x_{0})\cap L)- \{x_{0}\}$ the set $S_{2r}(x)=\exp_{x}(B_{2r}(0)\cap\nu_{x}L)$ is an embedded submanifold.
\item $S_{r}(x)\cap S_{r}(x_{0})=\emptyset$ for $x\in (B_{2r}(x_{0})\cap L)- \{x_{0}\}.$
\item $L\cap S_{2r}(x_{0})=\{x_{0}\}.$
\end{itemize}

Let $x\in L$ and $\xi\in\nu_{x}L$ with $\| \xi\|<r$. Set $\gamma: t\rightarrow \exp_{x}(t\xi)$. Then the choice of $r$, the equifocality of $\F$ and the fact that the leaves of $\F$ are closed embedded imply that $\gamma:[0,1+\epsilon)\rightarrow M$ is a minimal segment of geodesic. Hence $\gamma(1)$ is not conjugate to $\gamma(0)$. This allow us to conclude that $\exp_{x}:B_{r}(0)\cap\nu_{x}L\rightarrow M$ is an immersion.  The choice of $r$, the equifocality of $\F$ and the fact that the leaves of $\F$ are closed embedded also imply that the immersion $\exp_{x}:B_{r}(0)\cap\nu_{x}L\rightarrow M$ is injective and that $S_{r}(x)\cap S_{r}(y)=\emptyset$ if $x\neq y$. This finishes the proof of Claim 1.

It follows from Claim 1  that $\pi: M\rightarrow M/\F$ is an open map.

By Proposition \ref{cor-transnormal-system} and the same arguments of Claim 1 we can deduce the next claim.

\claim 2: \emph{For each $L$ we can find a neighborhood $U_{L} \subset \Sigma$ of $L$  and $r_{L}$ such that $\tub_{r_{L}}(U_{L})$ is a geometric tube.} 

Set $U:=\cup_{L\in \Sigma} \tub_{r_{L}}(U_{L})$. Note that $U$ is $\F$-invariant and $\exp$ is a local diffeomorphism between a neighborhood $\tilde{U}$ of $\Sigma$ in $\nu\Sigma$ and $U$. 

Now we define $\nu\Sigma/\F$ as the quotient of $\nu\Sigma$ with the following relation. Consider $\xi_{1},\xi_{2}\in\nu\Sigma$ and set $t\rightarrow \gamma_{i}(t)=\exp(t\xi_{i})$. We say that $[\xi_{1}]=[\xi_{2}]$ if there exists $\epsilon>0$ such that $\gamma_{1}(\epsilon)$ and $\gamma_{2}(\epsilon)$ are in the same leaf and one can transport $\gamma_{1}$ to $\gamma_{2}$ by parallel transport with respect to Bott connection. 

Claim 2 implies that the projection $\pi_{\Sigma}:\nu\Sigma\rightarrow\nu\Sigma/\F$ restricted to a neighborhood of $\Sigma$ in $\nu\Sigma$ is an open map. This fact and the fact that $[t \xi_{1}]=[t \xi_{2}]$  if $[ \xi_{1}]=[ \xi_{2}]$ imply that $\pi_{\Sigma}:\nu\Sigma\rightarrow\nu\Sigma/\F$ is an open map. 

Define $\exp_{\F}:\nu\Sigma/\F\rightarrow M/\F$ as $\exp_{\F}\circ\pi_{\Sigma}=\pi\circ\exp$. Since $\pi_{\Sigma}$, $\pi$ are open maps and
$\exp:\tilde{U}\rightarrow U$ is local diffeomorphism, we infer that $\exp_{\F}:\pi_{\Sigma}(\tilde{U})\rightarrow \pi(U)$ is a local homeomorphism. 

Since $\Sigma/\F$ is closed, it follows from a classical result of topology (see Kosinski \cite[Lemma I.7.2]{Kosinski}) that there exists a neighborhood $V_{1}$ of $\Sigma/\F$ in $\nu\Sigma/\F$ and $V_{2}$ of $\Sigma/\F$ in $M/\F$ such that $\exp_{\F}:V_{1}\rightarrow V_{2}$ is a  homeomorphism. Finally set $V_{3}=\pi^{-1}(V_{2})$ and define $V:=V_{3}\cap U.$

\end{proof}

%%%%%%%%%%%%%%%%%%%%%%%%%%%%%%%% SECTION PROOF OF THEOREM
%%%%%%%%%%%%%%%%%%%%%%%%%%%%%%%

\section{Proof of Theorem \ref{thm-Blowup-srf}}

The construction of the desired  metric $\hat{g}_{r}$  on a blow-up space $\hat{M}_{r}(\Sigma)$  will require several steps. 

The first step is the construction of a metric $\tilde{g}$ on a neighborhood of $\Sigma$ with properties similar to the properties of the metric defined in Proposition \ref{lemma-almost-product}. 

%In order to get this metric we will need a distribution $\mathcal{V}$ that is tangent to the leaves of $\F$.

\begin{prop}
\label{prop-metrica-adaptada-tub}
Let $\Sigma$ be the minimal stratum (with leaves of dimension $k_{0}$). 
Then there exists  a neighborhood  $\tub_{r}(\Sigma)$ of $\Sigma$   and a metric  $\tilde{g}$ on $\tub_{r}(\Sigma)$   such that:
\begin{enumerate}
\item[(a)] $\F$ restrict to  $\tub_{r}(\Sigma)$ is a Riemannian foliation with respect to $\tilde{g}$.
\item[(b)]The associated transverse metric is not changed, i.e., the distance between the plaques with respect to $g$ is the same distance between the plaques with respect to $\tilde{g}$.
\item[(c)] If a curve $\gamma$ is a unit speed segment of geodesic  orthogonal to $\Sigma$ with respect to the original metric $g$, then 
$\gamma$ is a unit speed segment of geodesic  orthogonal to $\Sigma$ with respect to the new metric $\tilde{g}$.
\item[(d)] There exists a smooth distribution $H$, whose dimension is equal to the codimension of the leaves of $\F|_{\Sigma}$, such that the normal space  of each plaque of $\F|_{\tub_{r}(\Sigma)}$ (with respect to $\tilde{g}$) is contained in $H.$
\item[(e)] The restriction of the metric $\tilde{g}$ to the stratum $\Sigma$ coincides to the restriction of the original metric $g$ to $\Sigma.$ 
\end{enumerate}
\end{prop}
\begin{proof}

First note that the proof of Proposition  \ref{lemma-almost-product} works if we replace $T P^{2}$ by a possible nonintegrable distribution $\mathcal{P}$  so that $\mathcal{P}$ is always tangent to the leaves. 

Now consider the metric $g^{0}$ and distribution $H$ of Proposition \ref{prop-flat-metric-N} and define $\mathcal{P}$ as the orthogonal space (with respect to $g^{0}$) to $H$. Note that $\mathcal{P}$ is always tangent to the leaves of $\F$.

Using the fact that  $d \pi|_{\mathcal{P}}: \mathcal{P}\rightarrow P_{q}$ is an isomorphism, we can define a metric on  $\mathcal{P}$ as
$g^{2}:=(d\pi)^{*}g.$

Let $D$ be the normal distribution  to $\mathcal{P}$ with respect to the original metric $g$ and define $\Pi:T_{p}M\rightarrow D_{p}$ as the orthogonal projection with respect to $g$. Note that $\Pi|_{H}:H_{p}\rightarrow D_{p}$ is an isomorphism. We define $g^{1}_{\alpha}:= (\Pi|_{H})^{*}g$ and $\tilde{g}:=g^{1}+g^{2}$, meaning that $\mathcal{P}$ and the distribution $H$ meet orthogonally (with respect to $\tilde{g}$).  

Finally we can repeat the same arguments of the proof of Proposition  \ref{lemma-almost-product} to get the desired result.

\end{proof}

%%%%%%%%%%%%%%%%%%%%%%%%%%%%%%%%%%%%
We come now to the second step of our construction, 
which is to change the metric $\tilde{g}$ in 
some directions, getting a new metric $\hat{g}^{M}$ on $\tub_{r}(\Sigma)-\Sigma$. 

%%%in the second step of our construction we modify.. 

First note that for small $\xi\in \nu_{q}\Sigma$ we can decompose $T_{\exp_{q}(\xi)}M$ as a direct sum of orthogonal subspaces (with respect to the metric $\tilde{g}$ of 
Proposition \ref{prop-metrica-adaptada-tub}) %as follows:
\[T_{\exp_{q}(\xi)}M= H_{\exp_{q}(\xi)}^{\perp}\oplus H_{\exp_{q}(\xi)}^{1}\oplus H_{\exp_{q}(\xi)}^{2}\oplus H_{\exp_{q}(\xi)}^{3},\]
where $H_{\exp_{q}(\xi)}^{\perp}$ is orthogonal to $H_{\exp_{q}(\xi)}$ and  $H_{\exp_{q}(\xi)}^{i}\subset H_{\exp_{q}(\xi)}$ is defined below.
\begin{enumerate}
\item $H_{\exp_{q}(\xi)}^{1}$ is the line generated by $\frac{d}{d t}\exp_{q}(t\xi)|_{t=1}$. 
\item Set $T_{q}=\exp_{q}(\nu(\Sigma)\cap B_{\epsilon}(0))$. Then $H_{\exp_{q}(\xi)}^{2}\subset T_{\exp_{q}(\xi)}T_{q}$ and orthogonal to $H_{\exp_{q}(\xi)}^{1}$.
\item $H_{\exp_{q}(\xi)}^{3}$ is orthogonal to $T_{\exp_{q}(\xi)}T_{q}$. 
\end{enumerate}
%$H_{\exp(\xi)}^{1}$ is the line generated by $\frac{d}{d t}\exp_{q}(t\xi)|_{t=1}$, 
%$H_{\exp(\xi)}^{2}\subset T_{\exp(\xi)}T$ and orthogonal to $H_{\exp(\xi)}^{1}$ and $H_{\exp(\xi)}^{3}$ is orthogonal to $T_{\exp(\xi)}T$. 

Now we define a new metric $\hat{g}^{M}$ on $\tub_{r}(\Sigma)-\Sigma$ as follows:
\begin{equation}
\label{eq-metrica-gM}
\hat{g}^{M}_{\exp_{q}(\xi)}(V,W):=\tilde{g}(V_{\perp},W_{\perp})+ \tilde{g}(V_{1},W_{1})+\frac{1}{\|\xi\|^{2}} \tilde{g}(V_{2},W_{2})+  \tilde{g}(V_{3},W_{3}),
\end{equation}
where $V_{i},W_{i}\in H_{\exp_{q}(\xi)}^{i}$ and  $V_{\perp},W_{\perp}\in H_{\exp_{q}(\xi)}^{\perp}.$

%%%%%%%%%%%%%%%%%%%%%% em obras
\begin{prop}
\label{prop-metric-on-N-L}
$\F$ is a s.r.f on $\tub_{r}(\Sigma)-\Sigma$ with respect to $\hat{g}^{M}.$ In addition if $\gamma:[0,a]\rightarrow \tub_{r}(\Sigma)$ is a unit speed geodesic orthogonal to $\Sigma$ with respect to the original metric $g$, then $\gamma|_{(0,a]}$ is a unit speed  geodesic with respect to $\hat{g}^{M}.$
\end{prop}
\begin{proof}

Let $P$ be a plaque and $\Sigma_{P}$ the stratum that contains $P$. Note that if $Y$ is a foliated vector field along $P$ tangent to $\Sigma_{P}$ and $H_{i}$ at a point $x$, then it follows from Proposition \ref{cor-transnormal-system}  that $Y$ is always tangent to $H_{i}$ and $\Sigma_{P}$. Also note that the function $\frac{1}{\|\xi\|^{2}}$ is constant along $P$ (see Proposition \ref{homothetic-lemma-stratum}). These two facts and  Proposition \ref{prop-metrica-adaptada-tub} imply that 
the Lie derivative $L_{X}(\hat{g}^{M})_{T}$ is zero. We can use the same argument to conclude that 
the Lie derivate of the transverse metric $(\hat{g}^{M})_{T}$ along each other plaque of $\Sigma_{P}$ is also zero and 
 hence, by Proposition \ref{prop-equivalencia-df-FR}, $\F|_{\Sigma}$ is a Riemannian foliation with respect to $\hat{g}^{M}$. Therefore, by Proposition \ref{prop-metric-stratum}, $\F$ is a s.r.f on  $\tub_{r}(\Sigma)-\Sigma$ with respect to $\hat{g}^{M}.$

Now we prove that $\gamma|_{(0,a]}$ is a geodesic with respect to $\hat{g}^{M}.$ It suffices to prove that for each $t_{0}\in(0,a]$ there exists $\epsilon>0$ such that $\gamma_{[t_{0}-\epsilon,t_{0}]}$ is a geodesic. Suppose that this is not true. 
 Since $\gamma$ is a horizontal geodesic with respect to $\tilde{g}$ (see Proposition  \ref{prop-metrica-adaptada-tub}) all the leaves $L_{\gamma(t)}$ belong to the same stratum $\Sigma_{\gamma(a)}$ for $t\in(0,a]$.  Then for small $\epsilon$   there exists a segment of horizontal geodesic (with respect to to $\hat{g}^{M}$) $\beta\subset \Sigma_{\gamma(a)}$ that joins $L_{\gamma(t_{0})}$ to $L_{\gamma(t_{0}-\epsilon)}$ and so that 

\begin{eqnarray*}
\hat{l}^{M}(\beta)&< & \hat{l}^{M}(\gamma|_{[t_{0}-\epsilon,t_{0}]})\\
                                   &=& \tilde{l}(\gamma|_{[t_{0}-\epsilon,t_{0}]}).
\end{eqnarray*} 
where $\tilde{l}(\cdot)$ and $ \hat{l}^{M}(\cdot)$ denote the length of the curve with respect to the metrics $\tilde{g}$ and  $\hat{g}^{M}$. On the other hand, 
\begin{eqnarray*}
\tilde{l}(\gamma|_{[t_{0}-\epsilon,t_{0}]})&< &\tilde{l}(\beta)\\
                                   &\leq& \hat{l}^{M}(\beta).
\end{eqnarray*} 
We arrived at a contradiction, and hence $\gamma_{[t_{0}-\epsilon,t_{0}]}$ is a geodesic. 

\end{proof}

%%%%%%%%%%%%%%%%%%%%% BLOW-UP
In the third step of our construction we pullback the metric $\hat{g}^M$ to the blow-up of $\tub_{r}(\Sigma)$ along $\Sigma$ (denoted here by $\hat{N}$) and then prove that the induced foliation $\hat{\F}$ on $\hat{N}$ is a s.r.f with respect to this new metric.
We start by recalling the  definition of blow-up  along a submanifold.

\begin{prop}
\label{prop-blowup-N}

Let $\tub_{r}(\Sigma)$ be the neighborhood of the minimal stratum $\Sigma$ defined in  Proposition \ref{prop-metrica-adaptada-tub}. Then
\begin{enumerate}
\item[(a)] $\hat{N}:=\{(x,[\xi])\in \tub_{r}(\Sigma)\times \mathbb{P}(\nu \Sigma)| x\in \exp^{\perp}(t\xi)\}$ is a smooth manifold (called blow-up of $\tub_{r}(\Sigma)$ along $\Sigma$) and $\hat{\pi}:\hat{N}\rightarrow \tub_{r}(\Sigma)$, defined as $\hat{\pi}(x,[\xi])=x$ is also smooth.
\item[(b)]$\hat{\Sigma}:=\hat{\pi}^{-1}(\Sigma)=\{(\rho([\xi]),[\xi])\in \hat{N}\}=\mathbb{P}(\nu\Sigma)$, where $\rho: \mathbb{P}(\nu \Sigma)\rightarrow \Sigma$ is the canonical projection.
\item[(c)] There exists a  singular foliation $\hat{\F}$ on $\hat{N}$ so that  $\hat{\pi}: (\hat{N}-\hat{\Sigma}, \hat{\F})\rightarrow (\tub_{r}(\Sigma)-\Sigma, \F)$ is a foliated diffeomorphism.
\end{enumerate}
\end{prop}
\begin{proof}
The proofs of items (a) and (b) are standard. In order to prove item (c) one can use the equifocality of $\F$ (see Theorem \ref{thm-s.r.f.-equifocal}) and the same argument used for blow-up of isometric actions (see Duistermaat and Kolk \cite[Section 2.9]{Duistermaat}).
\end{proof}

%%%%%%%%%%%%%%%%

\begin{prop}
\label{prop-metric-on-hatN}
Consider the  manifold $\hat{N}$ and the map $\hat{\pi}:\hat{N}\rightarrow \tub_{r}(\Sigma)$  defined in Proposition \ref{prop-blowup-N}.
Then there exists a metric $\hat{g}$ on $\hat{N}$ 
 with the following properties:
\begin{enumerate}
\item[(a)] if a unit speed geodesic $\hat{\gamma}$ is orthogonal to $\hat{\Sigma}$ with respect to $\hat{g}$, then $\hat{\pi}(\hat{\gamma})$ is a unit speed  geodesic orthogonal to $\Sigma$ with respect to the original metric $g$.  
\item[(b)] $\hat{\pi}|_{\hat{\Sigma}}: (\hat{\Sigma},\hat{g})\rightarrow(\Sigma,g)$ is a Riemannian submersion. 
\item[(c)] $(\hat{N},\hat{\F},\hat{g})$ is a s.r.f.
\item[(d)]$(\hat{\Sigma},\hat{\F}|_{\hat{\Sigma}},\hat{g})$ is a s.r.f 
and the liftings of horizontal geodesics of $(\Sigma,\F|_{\Sigma},g)$ are horizontal geodesics of $(\hat{\Sigma},\hat{\F}|_{\hat{\Sigma}},\hat{g})$.
\end{enumerate}
\end{prop}
\begin{proof}
Since $\hat{\pi}: \hat{N}-\hat{\Sigma}\rightarrow \tub_{r}(\Sigma)-\Sigma$ is  diffeomorphism, we can define a metric on $\hat{N}-\hat{\Sigma}$ as $\hat{g}:=\hat{\pi}^{*}\hat{g}^{M}$. We want to define a metric along $\hat{\Sigma}$. 

Let $\xi$ be a vector of $\nu\Sigma$ with $\|\xi\|$=1. Consider the curve $t\rightarrow \hat{\gamma}(t):=(\exp(t\xi),[\xi])$ on $\hat{N}$ and define the metric $\hat{g}$ so that $\hat{\gamma}'(0)$ is orthogonal to $\hat{\Sigma}$ and $\hat{g}(\hat{\gamma}'(0),\hat{\gamma}'(0))=1.$ 

Note that a vector $\hat{V}_{[\xi]}\in T_{[\xi]}\hat{\Sigma}$ is the radial projection on $\hat{\Sigma}$ of a vector
$V_{t}\in T_{\hat{\gamma}(t)}\hat{\pi}^{-1}( \partial \tub_{t}(\Sigma))$ and define  
 $\hat{g}(\hat{V}_{[\xi]},\hat{W}_{[\xi]}):=\lim_{t\rightarrow 0} \hat{g}(V_{t},W_{t}).$ 

%Note that a vector 
%$\hat{V}_{[\xi]}\in T_{[\xi]}\hat{\Sigma}$ can be identified with a vector $V_{t}\in T_{\exp(t\xi/\|\xi\|)} \partial \tub_{t}(\Sigma)$.  
%We define $\hat{g}(\hat{V}_{[\xi]},\hat{W}_{[\xi]}):=\lim_{t\rightarrow 0} \hat{g}^{M}(V_{t},W_{t}).$ 
%Consider the curve $t\rightarrow \hat{\gamma}(t):=(\exp(t\xi),[\xi])$ on $\hat{N}$ and define the metric $\hat{g}$ so that $\hat{\gamma}'(0)$ is %orthogonal to $\hat{\Sigma}$ and $\hat{g}(\hat{\gamma}'(0),\hat{\gamma}'(0))=\tilde{g}(\xi,\xi).$ 

%It is not difficult to prove that $\hat{g}$ is well defined and smooth. 

%%%%%%%%%%%%%%%%%%%%%%%%%%%%%%%%% COLOCAR PROVA!
\begin{lemma}
\label{lemma0-prop-metric-on-hatN}
 $\hat{g}$ is well defined and smooth. 
\end{lemma}
\begin{proof}
It is not difficult to check that $\hat{g}$ is well defined. In order to prove that $\hat{g}$ is smooth we must find a smooth local frame $\{\hat{e}_{i}\}$ in a neighborhood of a point $[\hat{\xi}_{0}]\in \hat{\Sigma}=\mathbb{P}(\nu\Sigma)$ and show that $\hat{g}(\hat{e}_{i},\hat{e}_{j})$ is smooth.  Set $q_{0}:=\hat{\pi}([\hat{\xi}_{0}]).$ Define smooth linearly independent  vector fields  $e_{1},\ldots, e_{k}$ orthogonal to the foliation $\{T\}$ where $T_{q}:=\exp_{q}(\nu(\Sigma)\cap B_{\epsilon}(0))$ for $q\in\Sigma$ near $q_{0}.$ One can construct  smooth vector fields $\hat{e}_{1},\ldots,\hat{e}_{k}$ in a neighborhood of $[\hat{\xi}_{0}]$ such that $d\hat{\pi}\hat{e}_{i}=e_{i}$. This can be done using  local coordinates of $\hat{N}$ and $\tub_{r}(\Sigma)$ and the fact that for each smooth function $a$ with $a(0,\theta)=0$ we can find a smooth function $b$ such that $a(r,\theta)=r b(r,\theta).$ We note that 
$
\hat{g}_{\hat{x}}(\hat{e}_{i},\hat{e}_{j})=\hat{g}^{M}_{\hat{\pi}(\hat{x})}(e_{i},e_{j})=\tilde{g}_{\hat{\pi}(\hat{x})}(e_{i},e_{j})
$ is smooth for $1\leq i,j\leq k$. 
Now one can find a smooth linearly independent vector fields $\hat{e}_{k+1},\ldots, \hat{e}_{n-1}$ in a neighborhood of $[\hat{\xi}_{0}]$ such that
$d\hat{\pi}\hat{e}_{\alpha}=d(\exp_{q})_{r\xi}(rv_{\alpha}^{\xi})$ where $v_{\alpha}^{\xi}\in\nu_{q}\Sigma$ is orthogonal to $\xi\in\nu_{q}\Sigma$ and depends smoothly on $\xi$, which is  near $\xi_{0}.$ We conclude that 
$\hat{g}_{\hat{x}}(\hat{e}_{\alpha},\hat{e}_{\beta})=\tilde{g}_{\hat{\pi}(\hat{x})}(d(\exp_{q})_{r\xi}(v_{\alpha}^
{\xi}),d(\exp_{q})_{r\xi}(v_{\beta}^{\xi}))$ is smooth for $k+1\leq \alpha, \beta<n-1$.
By construction, $\hat{g}(\hat{e}_{i},\hat{e}_{\alpha})=0$ for $1\leq i\leq k$ and $k+1\leq \alpha\leq n-1$. 
Finally one can define $\hat{e}_{n}$ as a vector field such that $d\hat{\pi}\hat{e}_{n}$ is tangent to unit speed geodesics orthogonal to $\Sigma$. It follows from the construction that $\hat{g}(\hat{e}_{n},\hat{e}_{n})=1$ and $\hat{g}(\hat{e}_{n},\hat{e}_{j})=0$ for $j\neq n$.

\end{proof}

Item (a) follows direct from Proposition \ref{prop-metric-on-N-L} and item (b) is proved below.

\begin{lemma}
\label{lemma1-prop-metric-on-hatN}
 $\hat{\pi}|_{\hat{\Sigma}}: (\hat{\Sigma},\hat{g})\rightarrow(\Sigma,g)$ is a Riemannian submersion. 
\end{lemma}
\begin{proof}

Let $v\in T_{q}\Sigma$ be a unitary vector with respect to $g$ and recall that $g|_{\Sigma}=\tilde{g}|_{\Sigma}.$ Let $\gamma$ be a geodesic orthogonal to  $\Sigma.$ Consider a  unitary vector field $t\rightarrow v(t)$ along $\gamma$   orthogonal to $T_{q}=\exp_{q}(\nu_{q}\Sigma\cap B_{\epsilon}(0))$ (with respect to $\tilde{g}$) so that $v(0)=v$. Let $\hat{\gamma}$ be the horizontal geodesic in $\hat{N}$ such that $\hat{\pi}(\hat{\gamma})=\gamma$ and $\hat{v}$ the vector field along $\hat{\gamma}$ such that $d\hat{\pi}\hat{v}=v.$ Note that $\hat{v}(0)$ is tangent to $\hat{\Sigma}$ and is orthogonal to $\hat{\pi}|_{\hat{\Sigma}}^{-1}(q)$.
Finally note that 
$ d\hat{\pi}\hat{v}(0)=\lim_{t\rightarrow 0}d\hat{\pi}\hat{v}(t)=v$ and 
$\hat{g}(\hat{v}(0),\hat{v}(0))=\lim_{t\rightarrow 0}\hat{g}(\hat{v}(t),\hat{v}(t))=1=g(v).$ The last two equations imply the result.
\end{proof}

Item (c) is proved in the next lemma. 
\begin{lemma}
\label{lemma2-prop-metric-on-hatN}
 $\hat{\F}$ is a s.r.f on $\hat{N}$ with respect to $\hat{g}.$ 

\end{lemma}

\begin{proof}
Proposition \ref{prop-metric-on-N-L} implies that $\hat{\F}$ is a s.r.f on $\hat{N}-\hat{\Sigma}$. Hence we have to prove that $\hat{\F}$ is a s.r.f in a  neighborhood of a point $\hat{q}\in\hat{\Sigma}$. In what follows we will prove that  item (d) of  Proposition \ref{rem-prop-metric-stratum} is satisfied. 

The fact that  $\hat{\F}$ is a s.r.f on $\hat{N}-\hat{\Sigma}$, Proposition \ref{rem-prop-metric-stratum} and radial projection on $\hat{C}:=\hat{\pi}^{-1}(\partial \tub_{\epsilon}(\Sigma))$ imply that $(\hat{\F},\hat{C},\hat{g})$ is a s.r.f.  The fact that $(\hat{\F},\hat{C},\hat{g})$ is a s.r.f and the radial projection on $\hat{\Sigma}$ imply the next claim.

\claim 1: \emph{$(\hat{\Sigma},\hat{\F}|_{\hat{\Sigma}})$ is a s.r.f with respect to some metric.}

Let $\hat{g}_{T}$ be the transverse metric (restrict to a stratum of $\hat{\Sigma}$). Note that the Lie derivative $L_{X}\hat{\metric}_{T}$ is zero for each $X\in \singularF$ tangent to $\hat{\Sigma}$, because each stratum of  
$(\hat{\Sigma},\hat{\F}|_{\hat{\Sigma}})$ is the intersection of the stratum of $\hat{\F}$ with $\hat{\Sigma}$ and $L_{X}\hat{\metric}_{T}$ is zero outside the manifold $\hat{\Sigma}$. This fact, Claim 1 and Proposition \ref{prop-metric-stratum} imply 

\claim 2: \emph{$(\hat{\Sigma},\hat{\F}|_{\hat{\Sigma}},\hat{g})$ is a s.r.f.} 

Consider a vector $v\in T_{\hat{q}}\hat{N}$ orthogonal
% $v$ is not tangent to $T_{q}\Sigma$ and $v$ is orthogonal 
to $\hat{L}_{\hat{q}}$. We also suppose that $v$ is not tangent to the stratum that contains $\hat{q}$. 
Let $v_{\hat{\Sigma}}\in T_{\hat{q}}\hat{\Sigma}$ so that $v=v_{\hat{\Sigma}}+ k \hat{\gamma}'(0)$ (for a real number $k$). Since $v$ and $\hat{\gamma}'(0)$ are orthogonal to $\hat{L}_{\hat{q}}$ we conclude that

\claim 3: \emph{$v_{\hat{\Sigma}}$ is orthogonal to $\hat{L}_{\hat{q}}$.}

Claims 2 and 3 together with the fact that $v_{\hat{\Sigma}}$ is not tangent to the stratum (in $\hat{\Sigma}$) that contains $\hat{q}$ imply 

\claim 4: \emph{There exists a distribution $t\rightarrow V_{t}$ along a curve $\beta\subset \hat{\Sigma}$ that satisfies Item (d.1) of Proposition \ref{rem-prop-metric-stratum} and so that $v_{\Sigma}$ is orthogonal to $V_{0}$.}

Note that Item (d.2) of Proposition \ref{rem-prop-metric-stratum} is also satisfied, i.e., $v$ is orthogonal to $V_{0}$, since $V_{0}\subset T_{\hat{q}}\hat{\Sigma}$ and $v_{\Sigma}$ is orthogonal to $V_{0}$. 

It is not difficult to check the other items of Proposition \ref{rem-prop-metric-stratum} and hence this proposition guarantee that $\hat{\F}$ is a s.r.f in a neighborhood of $\hat{q}$.  
\end{proof}

In order to prove item (d) recall that we have already proved  in Claim 2 of  Lemma \ref{lemma2-prop-metric-on-hatN} that 
$(\hat{\Sigma},\hat{\F}|_{\hat{\Sigma}},\hat{g})$ is a s.r.f.

The fact that the distribution   $H^{\perp}$ is tangent to the leaves of $\F$ and the same arguments of Lemma \ref{lemma0-prop-metric-on-hatN} and 
Lemma \ref{lemma1-prop-metric-on-hatN} imply that, for $\hat{L}_{\hat{q}}\subset\hat{\Sigma}$, there exists a base $\hat{e}_{1},\ldots,  \hat{e}_{l}$ of $T_{\hat{q}}\hat{L}_{\hat{q}}$   such that
\begin{enumerate}
\item $\hat{e}_{1}, \ldots, \hat{e}_{k}\in \nu_{\hat{q}}\hat{\pi}^{-1}(q)$ and $d\hat{\pi}\hat{e}_{1},\ldots , d\hat{\pi}\hat{e}_{k} $ is a basis of $T_{q}L_{q}$ for $q=\hat{\pi}(\hat{q}).$
\item $\hat{e}_{k+1}, \ldots, \hat{e}_{l}\in T_{\hat{q}}\hat{\pi}^{-1}(q)$.
\end{enumerate}

Note that if $\hat{\alpha}$ is the lifting of a horizontal geodesic of $(\Sigma,\F|_{\Sigma},g)$, then $\hat{\alpha}'(0)$ is orthogonal to $\hat{e}_{i}$ for $i=1,\ldots, l$.  Therefore $\hat{\alpha}$ is orthogonal to $\hat{L}_{\hat{q}}$. Since  $(\hat{\Sigma},\hat{\F}|_{\hat{\Sigma}},\hat{g})$ is a s.r.f, $\hat{\alpha}$ is a horizontal geodesic of $(\hat{\Sigma},\hat{\F}|_{\hat{\Sigma}},\hat{g})$ and this finishes the proof of item (d).

\end{proof}

%By the same arguments used for blow-up of isometric actions (see Duistermaat and Kolk \cite[Section 2.9]{Duistermaat}).

In the last step of our construction, we   glue $\hat{N}$ with a copy of $M-\tub_{r}(\Sigma)$, and hence we  construct the space 
$\hat{M}_{r}(\Sigma)$ and  the projection $\hat{\pi}_{r}:\hat{M}_{r}(\Sigma)\rightarrow M$  of  Theorem \ref{thm-Blowup-srf}. 
We can also induced the singular foliation  $\hat{\F}_{r}$ on $\hat{M}_{r}(\Sigma)$.
This procedure is analogous to the one used for blow-up of isometric actions (see Duistermaat and Kolk \cite[Section 2.9]{Duistermaat}).  

We must define the appropriate metric   $\hat{g}_{r}$ on  $\hat{M}_{r}(\Sigma)$. We need  a partition of unity of $\hat{M}_{r}(\Sigma)$ by 2 functions $\hat{f}$ and $\hat{h}$ such that  
 
 \begin{enumerate}
 \item[(a)] $\hat{f}=1$ in $\tub_{r/2}(\hat{\Sigma})$ and $\hat{f}=0$ outside of $\tub_{r}(\hat{\Sigma})$.
 \item[(b)] $\hat{f}$ and $\hat{h}$ are constant in the cylinders  $\partial \tub_{\epsilon}(\hat{\Sigma})$ for $ \epsilon<2 r.$ 
 \end{enumerate}
 
 With these 2 functions we can define $\hat{g}_{r}:=\hat{f}\hat{g}+\hat{h} g$ and use Proposition \ref{prop-metric-stratum} and item (c) of Proposition \ref{prop-metric-on-hatN}  to prove that $\hat{\F}_{r}$ is Riemannian with respect to $\hat{g}_{r}.$  items (a) and (b) of Theorem \ref{thm-Blowup-srf} follow by  construction. Item (c) of Theorem \ref{thm-Blowup-srf} follows from the fact that $\hat{f}$ and $\hat{h}$ are constant along the cylinders and from item (a) of Proposition \ref{prop-metric-on-hatN}.
Finally item (b) and (d) of Proposition \ref{prop-metric-on-hatN} implies item (d) of Theorem \ref{thm-Blowup-srf}.

We conclude this section with a remark that will be useful in the next section. 

\begin{rem}
\label{rem-metrica-hatg-tildeg}
Let $\hat{\beta}:[0,a]\rightarrow \hat{M}_{r}(\Sigma)$ be a minimal  segment of horizontal geodesic. First assume that 
$\hat{L}_{\hat{\beta}(0)}, \hat{L}_{\hat{\beta}(a)}\subset \hat{M}_{r}(\Sigma)-\hat{\Sigma}.$ Then $\hat{\beta}$ does not meet $\hat{\Sigma}.$ In addition,
\begin{equation}
\label{eq-rem-metrica-hatg-tildeg}
\tilde{g}_{r}(d\hat{\pi}_{r}(\hat{\beta}'(t)),d\hat{\pi}_{r}(\hat{\beta}'(t)))\leq \hat{g}_{r}(\hat{\beta}'(t),\hat{\beta}'(t)),
\end{equation}
for the metric
 $\tilde{g}_{r}:=f\tilde{g}+h g$ where $f$ and $h$ are defined as $\hat{f}=f\circ\hat{\pi}_{r}$ and $\hat{g}=g\circ\hat{\pi}_{r}$.
Note that, by continuity,  equation (\ref{eq-rem-metrica-hatg-tildeg}) is also valid if 
$\hat{L}_{\hat{\beta}(0)}, \hat{L}_{\hat{\beta}(a)}\subset \hat{\Sigma}.$

%Let $\hat{\beta}$ be a minimal  segment of horizontal geodesic in $\hat{M}_{r}(\Sigma)$. Then $\hat{\beta}$ does not meet $\hat{\Sigma}.$ In addition, 
% $\tilde{g}_{r}(d\hat{\pi}_{r}(\hat{\beta}'(t)),d\hat{\pi}_{r}(\hat{\beta}'(t)))\leq \hat{g}_{r}(\hat{\beta}'(t),\hat{\beta}'(t)), $ for the metric
% $\tilde{g}_{r}:=f\tilde{g}+h g$ where $f$ and $h$ are defined as $\hat{f}=f\circ\hat{\pi}_{r}$ and $\hat{g}=g\circ\hat{\pi}_{r}$.

\end{rem} 

%%%%%%%%%%%%%%%%%%%%%%%%%%%%
%%%%%%%%%%%%%%%%%%%%%%%%%% SECAO 3: PROVA DO TEOREMA DE CONVERGENCIA DE GROMOV HAUSDORFF
%%%%%%%%%%%%%%%%%%%%%%%%%%

\section{Proof of Theorem \ref{cor-epsilon-isometria}}

In order to prove Theorem \ref{cor-epsilon-isometria} it suffices to prove the next result.

\begin{prop}
\label{prop-epsilon-isometricos}
Let $\F$ be a singular Riemannian foliation with compact leaves on a compact Riemannian manifold $(M,g)$
and  $\Sigma$ the minimal stratum of $\F$. Consider the Riemannian manifold  $(\hat{M}_{r}(\Sigma),\hat{g}_{r})$   and  the foliation $\hat{\F}_{r}$ defined in Theorem  \ref{thm-Blowup-srf}. Then for each small $\epsilon>0$ there exists $r$ such that 
 for each $\hat{q}, \hat{p}\in \hat{M}_{r}(\Sigma)$ we have 
\[ |d(L_{q},L_{p})-\hat{d}_{r}(\hat{L}_{\hat{q}},\hat{L}_{\hat{p}})|<\epsilon\]
where $p:=\hat{\pi}_{r}(\hat{p})$ and $q:=\hat{\pi}_{r}(\hat{q}).$
\end{prop}

%\begin{rem}
%Note that if the leaves of $\F$ are closed embedded and $M/\F$ is compact, then the $\F$-invariant neighborhood $V$ of $\Sigma$ can be chosen to be %$\tub_{r}(\Sigma)$.
%\end{rem}

 The proof of  Proposition \ref{prop-epsilon-isometricos}  will require the next two lemmas.

%%%%%%%%%%%%%%%%%%%%%%%%%%%% LEMA FUNDAMENTAL

\begin{lemma}
\label{lemma-fundamental-thm-gromov-hausdorff-convergence-sigma}

For  each small $\epsilon>0$ there exists $r$ so that if $\hat{\gamma}:[0,R]\rightarrow \hat{M}_{r}(\Sigma)$ is a unit speed minimal horizontal  geodesic  (with respect to $\hat{g}_{r}$)  then there exists a piecewise smooth horizontal curve $\hat{\gamma}^{\epsilon}:[0,R^{\epsilon}]\rightarrow  \hat{M}_{r}(\Sigma)$ and a partition $0=\tau_{1}<\ldots <\tau_{m}=R^{\epsilon}$ with the following properties:
\begin{enumerate}
\item[(a)] $|\hat{l}_{r}(\hat{\gamma}^{\epsilon})-\hat{l}_{r}(\hat{\gamma})|<\epsilon.$
\item[(b)] For each $[\tau_{i},\tau_{i+1}]$ one of the following conditions is fulfilled
\begin{enumerate}
\item[(1)] $\hat{\gamma}^{\epsilon}|_{[\tau_{i},\tau_{i+1}]}$ is a segment of $\hat{\gamma}$ and $\hat{\gamma}^{\epsilon}|_{[\tau_{i},\tau_{i+1}]}\cap \tub_{r}(\hat{\Sigma})=\emptyset$ or
\item[(2)]$\hat{\gamma}^{\epsilon}|_{[\tau_{i},\tau_{i+1}]}$ is a horizontal geodesic orthogonal to $\hat{\Sigma}$ or
\item[(3)]$\hat{\gamma}^{\epsilon}|_{[\tau_{i},\tau_{i+1}]}\subset \hat{\Sigma}$, 
$\hat{l}_{r}(\hat{\gamma}^{\epsilon}|_{[\tau_{i},\tau_{i+1}]})=d(L_{\hat{\pi}_{r}(\hat{\gamma}(\tau_{i}))},L_{\hat{\pi}_{r}(\hat{\gamma}(\tau_{i+1}))}).$
\end{enumerate}
\end{enumerate}
\end{lemma}
\begin{proof}
Items (c) and (d) of Theorem \ref{thm-Blowup-srf} and the fact that $M/\F$ and $\Sigma/\F$ are bounded imply that there exists $K>0$ so that, for each $r$ the diameter of $\hat{M}_{r}(\Sigma)/\hat{\F}_{r}$ is lower than $K$. In particular we have

\claim 1: \emph{$R<K$ for each $r$.} 

Let $\epsilon_{0}$ be a radius so that if $L_{p}, L_{q}\subset \Sigma$ and $d(L_{p},L_{q})<\epsilon_{0}$ then each minimal segment of horizontal geodesic that joins $L_{p}$ to $L_{q}$ is contained in $\Sigma.$

\claim 2:\emph{ For   $\epsilon_{1}<\epsilon_{0}/3$ we can find a small $r$ such that if $\hat{d}_{r}(\hat{L}_{\hat{p}},\hat{L}_{\hat{q}})<2\epsilon_{0}/3$ and $\hat{p}, \hat{q}\in \hat{\Sigma}$ then 
\[ |d(L_{\hat{\pi}_{r}(\hat{p})},L_{\hat{\pi}_{r}(\hat{q})})-\hat{d}_{r}(\hat{L}_{\hat{p}},\hat{L}_{\hat{q}})|< \epsilon_{1}. \]
}

In order to prove Claim 2, let $\hat{\beta}$ be a minimal segment of horizontal geodesic that joins $\hat{L}_{\hat{p}}$ to $\hat{L}_{\hat{q}}$. It follows from Remark \ref{rem-metrica-hatg-tildeg} that 
\begin{equation}
\label{eq-1-lemma-fundamental-thm-gromov-hausdorff-convergence-sigma}
\tilde{l}_{r}(\hat{\pi}_{r}(\hat{\beta}))\leq \hat{l}_{r}(\hat{\beta}).
\end{equation}

Given $\epsilon_{1}$ we can find  $r$, that does not depend on $\hat{\beta}$, so that
\begin{equation}
\label{eq-2-lemma-fundamental-thm-gromov-hausdorff-convergence-sigma}
l(\hat{\pi}_{r}(\hat{\beta}))\leq \tilde{l}_{r}(\hat{\pi}_{r}(\hat{\beta}))+\epsilon_{1}.
\end{equation}
In fact the above equation can be proved using Claim 1 and reducing $r$ in such a way that
 the distribution $D$ defined in Proposition \ref{prop-metrica-adaptada-tub} turns out to be close to the distribution $H$.

Therefore 
\begin{eqnarray*}
d(L_{\hat{\pi}_{r}(\hat{p})},L_{\hat{\pi}_{r}(\hat{q})}) &\leq & l(\hat{\pi}_{r}(\hat{\beta}))\\
                                                         &\leq & \tilde{l}_{r}(\hat{\pi}_{r}(\hat{\beta}))+\epsilon_{1}\\
                                                         &\leq & \hat{l}_{r}(\hat{\beta})+ \epsilon_{1},
\end{eqnarray*}
and hence
\begin{equation}
\label{eq-3-lemma-fundamental-thm-gromov-hausdorff-convergence-sigma}
d(L_{\hat{\pi}_{r}(\hat{p})},L_{\hat{\pi}_{r}(\hat{q})}) \leq \hat{d}_{r}(\hat{L}_{\hat{p}},\hat{L}_{\hat{q}})+ \epsilon_{1}<\epsilon_{0}.
\end{equation}

Item (d) of Theorem \ref{thm-Blowup-srf}, equation (\ref{eq-3-lemma-fundamental-thm-gromov-hausdorff-convergence-sigma}) and the definition of $\epsilon_{0}$ imply
\begin{equation}
\label{eq-4-lemma-fundamental-thm-gromov-hausdorff-convergence-sigma}
\hat{d}_{r}(\hat{L}_{\hat{p}},\hat{L}_{\hat{q}})\leq d(L_{\hat{\pi}_{r}(\hat{p})},L_{\hat{\pi}_{r}(\hat{q})}).
\end{equation}

Equations (\ref{eq-3-lemma-fundamental-thm-gromov-hausdorff-convergence-sigma}) and (\ref{eq-4-lemma-fundamental-thm-gromov-hausdorff-convergence-sigma}) imply
\begin{equation*}
0\leq d(L_{\hat{\pi}_{r}(\hat{p})},L_{\hat{\pi}_{r}(\hat{q})})-\hat{d}_{r}(\hat{L}_{\hat{p}},\hat{L}_{\hat{q}})\leq \epsilon_{1}
\end{equation*}
and this conclude the proof of Claim 2.

Now consider a integer $N_{0}$ such that $K<N_{0}\frac{\epsilon_{0}}{3}$ and a partition $0=t_{0}<\cdots < t_{N}=R$ so that $N\leq N_{0}$ and $t_{i}-t_{i-1}\leq \frac{\epsilon_{0}}{3}$. 

If $\hat{\gamma}|_{[t_{i},t_{i+1}]}\cap\tub_{r}(\hat{\Sigma})=\emptyset$ then we define   $\hat{\gamma}^{\epsilon}_{i}:= \hat{\gamma}([t_{i},t_{i+1}])$.

Now assume that $\hat{\gamma}([t_{i},t_{i+1}]) \cap \tub_{r}(\hat{\Sigma}) \neq \emptyset$. 
We will replace $\hat{\gamma}([t_{i},t_{i+1}])$ by a  piecewise smooth horizontal curve $\hat{\gamma}^{\epsilon}_{i}$ that fulfills item (b) and has length close to $\hat{\gamma}([t_{i},t_{i+1}])$.

Let $a_{i}$ be the smallest number in $[t_{i},t_{i+1}]$ such that $\hat{\gamma}(a_{i})\in\overline{\tub_{r}(\hat{\Sigma})}$ and $b_{i}$ the biggest number in $[t_{i},t_{i+1}]$ such that 
$\hat{\gamma}(b_{i})\in\overline{\tub_{r}(\hat{\Sigma})}$.
Let $\hat{\sigma}_{i}$ and $\hat{\sigma}_{i+1}$ be  the projection of $\hat{\gamma}(a_{i})$ and $\hat{\gamma}(b_{i})$ in $\hat{\Sigma}$. 

By triangle inequality we have

\claim 3:\emph{ $|\hat{d}_{r}(\hat{L}_{\hat{\sigma}_{i}},\hat{L}_{\hat{\sigma}_{i+1}})-\hat{d}_{r}(\hat{L}_{\hat{\gamma}(a_{i})},\hat{L}_{\hat{\gamma}(b_{i})})|\leq 2r$. In particular for $r<\epsilon_{0}/6$ we have $\hat{d}_{r}(\hat{L}_{\hat{\sigma}_{i}},\hat{L}_{\hat{\sigma}_{i+1}})<2\epsilon_{0}/3<\epsilon_{0}.$ }

Let $\hat{\beta}_{i}$ be a minimal segment of horizontal geodesic in $\hat{M}_{r}(\Sigma)$ that joins $\hat{L}_{\hat{\sigma}_{i}}$ to $\hat{L}_{\hat{\sigma}_{i+1}}.$ By item (d) of Theorem \ref{thm-Blowup-srf} and Claim 2 we can find a horizontal curve $\hat{\alpha}_{i}\subset \hat{\Sigma}$ joining $\hat{L}_{\hat{\sigma}_{i}}$ to $\hat{L}_{\hat{\sigma}_{i+1}}$ such that 
$\hat{l}_{r}(\hat{\alpha}_{i})=d(L_{\hat{\pi}_{r}(\hat{\sigma}_{i})},L_{\hat{\pi}_{r}(\hat{\sigma}_{i+1})} ).$ By Claim 2 we have

\[ |\hat{l}_{r}(\hat{\alpha}_{i})-\hat{l}_{r}(\hat{\beta}_{i})|<\epsilon_{1}.\]

Claim 3 implies
\[ | \hat{l}_{r}(\hat{\beta}_{i})-\hat{l}_{r}(\hat{\gamma}|_{[a_{i},b_{i}]})  |\leq 2r .\]

Therefore

\[ | \hat{l}_{r}(\hat{\alpha}_{i})- \hat{l}_{r}(\hat{\gamma}|_{[a_{i},b_{i}]}) | \leq \epsilon_{1}+2r .\]

Let $\hat{\delta}_{i}$ (respectively $\hat{\delta}_{i+1}$) be a segment of horizontal geodesic perpendicular to $\hat{\Sigma}$ that joins $\hat{L}_{\hat{\sigma}_{i}}$ to $\hat{L}_{\hat{\gamma}(a_{i})}$ (respectively  $\hat{L}_{\hat{\sigma}_{i+1}}$ to $\hat{L}_{\hat{\gamma}(b_{i})}$).

Define 
$\hat{\gamma}^{\epsilon}_{i}:= \hat{\gamma}([t_{i},a_{i}])\cup \hat{\delta}_{i}^{-1}\cup \hat{\alpha}_{i}\cup \hat{\delta}_{i+1}\cup \hat{\gamma}([b_{i},t_{i+1}])$ and note that
\[ | \hat{l}_{r}(\hat{\gamma}^{\epsilon}_{i}) -\hat{l}_{r}(\hat{\gamma}|_{[t_{i},t_{i+1}]})|<\epsilon_{1}+ 2r+2r .\]

Finally setting $\hat{\gamma}^{\epsilon}:=\hat{\gamma}^{\epsilon}_{1}\cup\cdots\cup \hat{\gamma}^{\epsilon}_{N}$ we have

\[ |\hat{l}_{r}(\hat{\gamma}^{\epsilon})-\hat{l}_{r}(\gamma)|< N(\epsilon_{1}+4r)\leq N_{0}(\epsilon_{1}+4r).\]

Since $N_{0}$ does not depend on $r$ or $\hat{\gamma}$ the above equation implies item (a) of the lemma and this conclude the proof.

\end{proof}

%%%%%%%%%%%%%%%%%%%%%%%%%%%%%%%%%%%%% LEMA SUPLEMENTAR

Using some arguments of the above lemma we can also prove the next result. 

\begin{lemma}
\label{lemma-suplementar-thm-gromov-hausdorff-convergence-sigma}

For  each small $\epsilon>0$ there exists $r$ so that if $\gamma:[0,R]\rightarrow M$ is a unit speed minimal horizontal  geodesic  (with respect to $g$)  then there exists a piecewise smooth horizontal curve $\gamma^{\epsilon}:[0,R^{\epsilon}]\rightarrow  M$ and a partition $0=\tau_{1}<\ldots <\tau_{m}=R^{\epsilon}$ with the following properties:
\begin{enumerate}
\item[(a)] $|l(\gamma^{\epsilon})-l(\gamma)|<\epsilon.$
\item[(b)] For each $[\tau_{i},\tau_{i+1}]$ one of the following conditions is fulfilled
\begin{enumerate}
\item[(1)] $\gamma^{\epsilon}|_{[\tau_{i},\tau_{i+1}]}$ is a segment of $\gamma$ and $\gamma^{\epsilon}|_{[\tau_{i},\tau_{i+1}]}\cap \tub_{r}(\Sigma)=\emptyset$ or
\item[(2)]$\gamma^{\epsilon}|_{[\tau_{i},\tau_{i+1}]}$ is a horizontal geodesic orthogonal to $\Sigma$ or
\item[(3)]$\gamma^{\epsilon}|_{[\tau_{i},\tau_{i+1}]}\subset \Sigma$, 
$l(\gamma^{\epsilon}|_{[\tau_{i},\tau_{i+1}]})=d(L_{\gamma(\tau_{i})},L_{\gamma(\tau_{i+1})}).$
\end{enumerate}
\end{enumerate}

\end{lemma}

%%%%%%%%%%%%%%%%%%%%%%%%%%%%%%%%%%% PROPOSICAO EPSILON- ISOMETRIA

We are now ready to prove Proposition \ref{prop-epsilon-isometricos}.
For a small $\epsilon$ consider $r$ defined in Lemmas \ref{lemma-fundamental-thm-gromov-hausdorff-convergence-sigma}
and \ref{lemma-suplementar-thm-gromov-hausdorff-convergence-sigma}.

Let $\hat{\gamma}$ be a minimal horizontal geodesic that joins $\hat{L}_{\hat{p}}$ to $\hat{L}_{\hat{q}}$ and $\hat{\gamma}^{\epsilon}$ the piecewise smooth horizontal curve defined in Lemma \ref{lemma-fundamental-thm-gromov-hausdorff-convergence-sigma}.
Note  that $l(\hat{\pi}_{r}(\hat{\gamma}^{\epsilon}))=\hat{l}_{r}(\hat{\gamma}^{\epsilon}).$ Hence 
 Lemma \ref{lemma-fundamental-thm-gromov-hausdorff-convergence-sigma}   implies
\begin{eqnarray*}
d(L_{p},L_{q})&\leq & l(\hat{\pi}_{r}(\hat{\gamma}^{\epsilon}))\\
              &= & \hat{l}_{r}(\hat{\gamma}^{\epsilon})\\
              &< & \hat{l}_{r}(\hat{\gamma})+\epsilon\\
              &= & \hat{d}_{r}(\hat{L}_{\hat{p}},\hat{L}_{\hat{q}})+\epsilon. 
\end{eqnarray*}
Now let $\gamma$ be a minimal horizontal geodesic that joins $L_{p}$ to $L_{q}$ and $\gamma^{\epsilon}$  the piecewise smooth horizontal curve defined in Lemma \ref{lemma-suplementar-thm-gromov-hausdorff-convergence-sigma}. By Theorem \ref{thm-Blowup-srf} there exists a curve 
$\hat{\gamma}$ in $\hat{M}_{r}(\Sigma)$ that joins $\hat{L}_{\hat{p}}$ to $\hat{L}_{\hat{q}}$ such that $\hat{l}_{r}(\hat{\gamma})= l(\gamma^{\epsilon})$. Hence Lemma \ref{lemma-suplementar-thm-gromov-hausdorff-convergence-sigma} implies
\begin{eqnarray*}
\hat{d}_{r}(\hat{L}_{\hat{p}},\hat{L}_{\hat{q}}) & \leq & \hat{l}_{r}(\hat{\gamma})\\
                                                 & =& l(\gamma^{\epsilon})\\
                                                 &<& l(\gamma)+\epsilon\\
                                                 & = & d(L_{p},L_{q})+\epsilon.
\end{eqnarray*}

%%%%%%%%%%%%%%%%%%%%%%%%%%%%%%%%%%%%%%%%%%%%%

\section{Appendix}

In this appendix  we recall that if $\rho:X\rightarrow Y$ is a surjective $\epsilon/2$-isometry between compact metric spaces, then $d_{G-H}(X,Y)<3\epsilon$. In particular, this implies  that Corollary \ref{cor-gromov-hausdorff-convergence} follows directly from Theorem \ref{cor-epsilon-isometria} (see also Burago, Burago and Ivanov \cite{Burago}).

We start by recalling the definition and some facts about  Gromov-Hausdorff distance (for details see  Petersen \cite{petersen}). 

\begin{dfn}
Let $X$ be a metric space and $A, B\subset X$. Then we define the \emph{Hausdorff distance} between $A$ and $B$ as  $d_{H}(A,B)=\inf\{ \epsilon: A\subset \tub_{\epsilon}(B), B\subset\tub_{\epsilon}(A)\}$. Now consider two metric spaces $X$ and $Y$ then an \emph{admissible metric} on the disjoint union $X\sqcup Y$ 
is a metric that extends the given metrics on $X$ and $Y$. With this we can define the \emph{Gromov-Hausdorff distance} as 
$d_{G-H}(X,Y)=\inf\{d_{H}(X,Y):$ admissible metrics on $X\sqcup Y\}$. 

\end{dfn}

It is possible to prove that if $X$ and $Y$ are compact metric spaces then $X$ and $Y$ are isometric if and only if $d_{G-H}(X,Y)=0$.
Since $d_{G-H}$ is symmetric and satisfies the triangle inequality, the collection of compact metric spaces $(\mathcal{M},d_{G-H})$ turns out to be a pseudometric space, and if we consider equivalence classes of isometric spaces it becomes a metric space. In fact, it is possible to prove that this metric space is complete and separable.

In what follows we will need a lemma about $\epsilon$-dense subsets. Recall that if $X$ is a compact metric space, then a finite subset $A\subset X$ is called \emph{$\epsilon$-dense subset} if every point of $X$ is within distance $\epsilon$ of some element in $A$, i.e., $d_{H}(A,X)<\epsilon.$

\begin{lemma}[Petersen \cite{petersen}]
\label{lemma-epsilon-denseSubset}

Suppose that we have $\epsilon$-dense subsets 
\[A=\{x_{1},\ldots, x_{k}\}\subset X\, \mathrm{and}\  B=\{y_{1},\ldots, y_{k}\}\subset Y\]
 with the further property that 
 \[|d(x_{i},x_{j})-d(y_{i},y_{j})|\leq \epsilon,  \  1\leq i,j\leq k.\]
 Then $d_{G-H}(X,Y)\leq 3\epsilon.$ 

\end{lemma}

\begin{prop}
Let $(X,d_{X})$ and $(Y,d_{Y})$ be compact metric spaces and $\rho:X\rightarrow Y$ a surjective $\epsilon/2$-isometry, i.e., 
\begin{equation}
\label{eq-provaTeoGromov-Hausdorff}
 |d_{Y}(\rho(x),\rho(\tilde{x}))-d_{X}(x,\tilde{x})|<\epsilon/2.
\end{equation}
Then $d_{G-H}(X,Y)<3\epsilon$, where $d_{G-H}$ is the distance of Gromov-Hausdorff.
\end{prop}

\begin{proof}
Consider $\{x_{1},\ldots,x_{k}\}$ an $\epsilon/2$-dense subset of $X$. Clearly $\{x_{1},\ldots,x_{k}\}$   is also an $\epsilon$-dense subset. Now set $y_{i}:=\rho(x_{i}).$ Then equation (\ref{eq-provaTeoGromov-Hausdorff}) implies that $\{ y_{1},\ldots, y_{k}\}$ is an $\epsilon$-dense subset of $Y$. Equation (\ref{eq-provaTeoGromov-Hausdorff}) and Lemma \ref{lemma-epsilon-denseSubset} then imply $d_{G-H}(X,Y)<3\epsilon.$
\end{proof}

%%%%%%%%%%%%%%%%%%%%%%%%%%%%%%%%%%%%%%%%%%%%%%%%%%%%%%%%%%%%%%%%%%%%%%%%%%%%%%%%
%%%%%%%%%%%%%%%%%%%%%%%%%%%% BIBLIOGRAPHY
%%%%%%%%%%%%%%%%%%%%%%%%%%%%%%%%%%%%%%%%%%%%%%%%
\bibliographystyle{amsplain}

\end{document}